\documentclass[11pt]{article}

\usepackage{amsmath, amssymb, amsthm, hyperref, enumerate}
\usepackage{graphicx, caption, subcaption}

\usepackage{titlesec}
\theoremstyle{plain}
\newtheorem{thm}{Theorem}[section]
\newtheorem{prop}[thm]{Proposition}
\newtheorem{lemma}[thm]{Lemma}
\newtheorem{sublemma}[thm]{Sublemma}

\newtheorem{cor}[thm]{Corollary}
\newtheorem{fact}[thm]{Fact}

\theoremstyle{definition}
\newtheorem{defn}[thm]{Definition}

\newtheorem{rmkx}[thm]{Remark}
\newenvironment{rmk}
{\pushQED{\qed}\rmkx}
{\popQED\endrmkx}

\newtheorem{examplex}[thm]{Example}
\newenvironment{example}
{\pushQED{\qed}\examplex}
{\popQED\endexamplex}

\newcommand{\barPH}{\overline{\mathbb{P}\mathcal{H}}}

\newcommand{\barM}{\overline{\mathcal{M}}}
\newcommand{\C}{\mathbb{C}}
\newcommand{\Z}{\mathbb{Z}}
\newcommand{\R}{\mathbb{R}}

\newcommand{\M}{\mathcal{M}}
\newcommand{\K}{\mathcal{K}}

\renewcommand{\H}{\mathcal{H}}
\renewcommand{\P}{\mathbb{P}}

\newcommand{\pert}{\operatorname{pert}}

\title{Measure bound for translation surfaces with short saddle connections}  
\author{Benjamin Dozier \thanks{Department of Mathematics, Cornell University, \href{mailto:benjamin.dozier@gmail.com}{\nolinkurl{benjamin.dozier@cornell.edu}}.}}
\begin{document}
\maketitle

\begin{abstract}
  We prove that any ergodic $SL_2(\R)$-invariant probability measure on a stratum of translation surfaces satisfies \emph{strong regularity}: the measure of the set of surfaces with two non-parallel saddle connections of length at most $\epsilon_1,\epsilon_2$ is $O(\epsilon_1^2 \cdot \epsilon_2^2)$.   We prove a more general theorem which works for any number of short saddle connections.  The proof uses the multi-scale compactification of strata recently introduced by Bainbridge-Chen-Gendron-Grushevsky-M\"oller and the algebraicity result of Filip. 
\end{abstract}

\setcounter{tocdepth}{1}  %
\tableofcontents
\setcounter{tocdepth}{4} %

\section{Introduction}
\label{sec:intro}

A \emph{translation surface} is a pair $(M,\omega)$, where $M$ is a Riemann surface, and $\omega$ is a holomorphic $1$-form on $M$ that is not identically zero.   This data determines a flat metric with singular points at the zeros of $\omega$.  The collection of all translation surfaces with the same genus and combinatorics of zeros form a \emph{stratum} $\H$, which admits a natural $SL_2(\mathbb{R})$-action.

Questions about the dynamics of the straight-line flow on individual surfaces are intimately connected to dynamics of this $SL_2(\mathbb{R})$-action on the space of surfaces.  The study of these dynamical systems has produced a flourishing field of research, bringing together techniques from Teichm\"uller theory, ergodic theory, homogeneous dynamics, and algebraic geometry.   Measures $\mu$ on $\H$ that are invariant under the $SL_2(\mathbb{R})$-action govern the dynamics.  This paper concerns the volume of sets of degenerating surfaces with respect to any such (ergodic) invariant measure.

A \emph{saddle connection} $s$ is a geodesic segment on a translation surface that connects two singular points (with no singular points in the interior).  We denote its length by $|s|$.  Even the (complex) projectivized stratum $\P\H$ is non-compact; a sequence of translation surfaces escapes every compact set iff there are saddle connections whose lengths go to zero.  A corollary of our main result is the following estimate on the volume of the set of surfaces with \emph{multiple} short saddle connections.    

\begin{thm}[Strong regularity]
  \label{thm:twosc}
  Let $\H$ be any stratum of translation surfaces, and let $\mu$ be any ergodic $SL_2(\mathbb{R})$-invariant probability measure on the unit area locus $\H^1\subset \H$.  Let 
  \begin{align*}
L_{\epsilon_1,\epsilon_2}:= &\{X \in \H^1 : X \text{ has non-parallel saddle connections } s_1,s_2, \\& \text{ with } |s_1|<\epsilon_1, |s_2|<\epsilon_2\}. 
  \end{align*}
  Then $$\mu(L_{\epsilon_1,\epsilon_2})  = O(\epsilon_1^2 \cdot \epsilon_2^2),$$
    where the implicit constant depends on $\mu$. 
  \end{thm}

  We will in fact prove a version of the above for any number of saddle connections, Theorem \ref{thm:regularity}.  The general theorem requires a more technical notion generalizing non-parallel, so we postpone its statement.  
  Our proof establishes a new paradigm for exploring flat geometry using the recently constructed multi-scale compactification of strata (due to Bainbridge-Chen-Gendron-Grushevsky-M\"oller \cite{bcggm2019}).  

\subsection{Previous work}
We call the above property of the measure $\mu$ \emph{strong regularity}.  A substantially weaker version of regularity was needed by Eskin-Kontsevich-Zorich \cite{ekz2014} as a crucial technical assumption in their proof of a striking formula relating sums of Lyapunov exponents of the Teichm\"uller geodesic flow to Siegel-Veech constants.  This weaker form was proven by Avila-Matheus-Yoccoz \cite{amy2013} using an intricate hands-on argument.  But the above stronger theorem is the natural bound that one expects based on independence heuristics.

Strong regularity has been proved for the Masur-Veech measure $\mu_{MV}$ (whose support is all of $\H$) by Masur-Smillie \cite[proof of Theorem 10.3]{ms1991}, and for measures coming from rank 1 affine invariant manifolds by Nguyen \cite{nguyen2012}.  The natural analogue of the above for the locus of surfaces with a single $\epsilon$-short saddle connection is a volume bound of $O(\epsilon^2)$.  This bound can be easily derived, for all measures $\mu$, from the Siegel-Veech formula.

\subsection{Two or more short saddle connections}

What we will actually prove is a generalization of Theorem \ref{thm:twosc} (Strong regularity) that works for any number of short saddle connections.

In order to state the generalization, we work in the context of an \emph{affine invariant submanifold} $\M$ of $\H$.  We will give a brief review of key concepts; for more background, the reader is encouraged to consult the survey \cite{wrightSurvey}.

An affine invariant manifold is defined to be an immersed suborbifold that is locally cut out by homogeneous linear equations in period coordinates with real coefficients.  An affine invariant manifold $\M$ comes with an \emph{affine measure} $\mu_\M$ supported on $\M$.  Locally, this measure equals Lebesgue measure on the linear subspace $T_X\M$ of period coordinates corresponding to $\M$. There is also a finite measure $\mu_\M^1$ supported on the unit area locus $\M^1 \subset \M$, defined by ``coning'':
$$\mu_\M^1 (S) := \mu_\M \left(\{  sX : X \in S, \text{ } 0\le s \le 1\}\right),$$
for any measurable subset $S\subset \H^1$.  We will take this measure to be normalized to have total mass $1$.  

The landmark result of Eskin-Mirzakhani \cite{em2018} states that any ergodic $SL_2(\mathbb{R})$-invariant probability measure on $\H^1$ is equal to $\mu^1_\M$ for some affine invariant manifold $\M$; these affine measures are thus of central importance.   Our generalization will be stated in terms of affine invariant manifolds, and then in Section~\ref{sec:deduction} we will use \cite{em2018} to deduce Theorem \ref{thm:twosc} (Strong regularity) from the general theorem.  

The next definition gives the appropriate analogue of non-parallel for any number of saddle connections.  Recall that the linear structure on $\H$ is given locally near $X$ by the relative cohomology $H^1(X,\Sigma;\C)$, where $\Sigma$ denotes the set of zeros of the differential.

\begin{defn}
  Saddle connections $s_1,\ldots,s_k$ on a surface $X\in \M$ are said to be \emph{$\M$-independent} if their relative homology classes define linearly independent functionals on the linear subspace $T_X\M \subset T_X\H \cong H^1(X,\Sigma; \C)$.\footnote{This is the definition we use at a point $X$ in the smooth locus $\M^*\subset \M$, since at such a point $T_X\M$ is a single linear subspace.  At points of $\M-\M^*$,  $T_X\M$ is not a single linear subspace, but rather a finite union of linear subspaces, so the above definition needs modification.  However $\mu_\M(\M-\M^*)=0$, so for the purposes of Theorem \ref{thm:regularity}, it does not matter how we define $\M$-independence at such $X$.  For concreteness, one can define \emph{any} set of saddle connections to be $\M$-independent at $X\in \M-\M^*$.} 
\end{defn}

We can now state the main theorem that we will prove.  

\begin{thm}
    \label{thm:regularity}
  Let $\M\subset \H$ be an affine invariant manifold, and let
  \begin{align*}
    L_{\epsilon_1,\ldots,\epsilon_k}^{\mathcal{M}} := &\{ X\in \M^1: X \text{ has } \M\text{-independent saddle connections } s_1,\ldots,s_k,\\
  & \text{ with } |s_i| \le \epsilon_i \}.
  \end{align*}
  Then
  $$\mu_\M^1 \left(L_{\epsilon_1,\ldots, \epsilon_k}^\M\right)  = O(\epsilon_1^2 \cdots \epsilon_k^2),$$
  where the implicit constant depends on $\M$. 
\end{thm}

Note that in contrast to Theorem \ref{thm:twosc}, the set $L_{\epsilon_1,\ldots,\epsilon_k}^{\mathcal{M}}$ we are looking at depends on $\M$.

\begin{rmk}
  \label{rmk:complex}
  The only place in the proof where we use that $\M$ is cut out by linear equations with \emph{real} coefficients (rather than arbitrary complex coefficients) is when we apply Filip's theorem \cite{filip2016} to get that $\M$ is an algebraic subvariety.  So the theorem above also holds for any subvariety of a stratum locally cut out by linear equations in period coordinates, with any complex coefficients.  In this setting, the existence of a globally defined linear measure $\mu$ supported on the locus is not immediate due to monodromy issues, but provided that one exists we can apply the theorem to it.  We will use this more general result in the proof of Theorem \ref{thm:kdiff} below.
\end{rmk}

\begin{rmk}
  The bound in Theorem \ref{thm:regularity} is \emph{sharp} in the sense that for any positive integer $k$, there exists $\M$ and a constant $c>0$ such that
  $$\mu_\M^1 \left(L_{\epsilon_1,\ldots, \epsilon_k}^\M\right)  \ge c \cdot \epsilon_1^2 \cdots \epsilon_k^2,$$
  when all the $\epsilon_i$ are sufficiently small.  Many constructions of surfaces with short saddle connections will achieve this; we will consider surfaces with several short slits which can be varied independently.  (See also \cite{ms1991} for related computations).  We take $\M$ to be the stratum $\H(1^{2k})$ (this notation means $2k$ simple zeros).  Let $X_k$ be a surface with a decomposition as in Figure \ref{fig:surfaceH12k}.  We consider a subset $\mathcal{S} \subset \M$ given by taking $X_k$ and varying by a small amount $\delta>0$ the side periods of the tori and the positions of the slits.  The $\delta$ is some fixed value independent of $\epsilon_1,\ldots,\epsilon_k$.  We then allow the slit given by $A_i,B_i$ to have any period of absolute value less than $\epsilon_i$, for each $i$.   Finally, we rescale all these surfaces to have area $1$.  This subset $\mathcal{S}$ is contained in $L_{\epsilon_1,\ldots, \epsilon_k}^\M$.   And
  \begin{align*}
    \mu^1_{\M}(\mathcal{S}) \ge c \cdot \epsilon_1^2 \cdots \epsilon_k^2,
  \end{align*}
  where $c$ depends only on $\delta$ (and hence not on the $\epsilon_i$).

\begin{figure}[]
\begin{center}
  \includegraphics[scale=0.83]{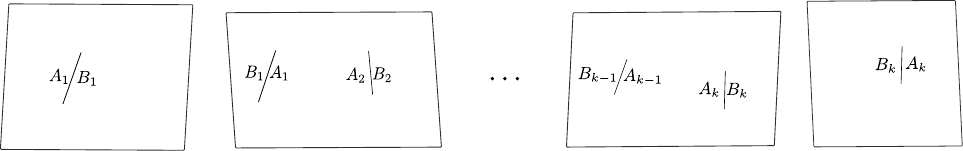}
  \caption{A surface $X_k$ in $\H(1^{2k})$.  The top and bottom sides of each parallelogram are glued together, as are the left and right sides. }
  \label{fig:surfaceH12k}
\end{center}
\end{figure}

\end{rmk}

Part of the motivation for proving the above theorem was as a test of our understanding of the structure of affine invariant manifolds near infinity.  Certain structural results about affine invariant manifolds are known, for instance Wright's Cylinder Deformation Theorem \cite{wright2015}.  But a full classification of affine invariant manifolds in all strata is still beyond reach.  The methods developed here, in particular the use of the multi-scale compactification to understand degenerations in affine invariant manifolds, should be helpful in this quest.

\subsection{Applications of Theorem \ref{thm:twosc} (Strong regularity)}
Gadre has proved several results about excursions of Teichm\"uller geodesics into the thin part of the moduli space, for geodesics chosen randomly according to some invariant measure $\mu$ \cite{gadre2017, gadre2019}.  These results were conditional on strong regularity of the measure, so they can now be given unconditional proofs.

There may be future applications of the theorem to counting problems for pairs of saddle connections.  For instance, the $L^2$-norm of the Siegel-Veech transform of a compactly supported indicator function is related to pairs of saddle connections.  Athreya-Cheung-Masur  \cite{acm2019} recently proved that this $L^2$-norm is finite for the case of the Masur-Veech measure $\mu_{MV}$; their proof uses strong regularity of $\mu_{MV}$ as well as some other special facts about this measure.  In the future we plan to investigate extending the result to any $\mu$ using Theorem \ref{thm:twosc} (Strong regularity).

\subsection{Deduction of Theorem \ref{thm:twosc} (Strong regularity) from Theorem \ref{thm:regularity}}
\label{sec:deduction}
This deduction involves applying Eskin-Mirzakhani \cite{em2018}, and then relating $\M$-independence of two saddle connections to the property of being non-parallel.  
\begin{proof}[Proof of Theorem \ref{thm:twosc} (Strong regularity) given Theorem \ref{thm:regularity}]
  By \cite[Theorem 1.4]{em2018}, $\mu=\mu^1_{\M}$ for some affine invariant manifold $\M$.  Let $\M^*\subset \M$ be the smooth locus (i.e. where the immersed suborbifold does not intersect itself).  We claim that 
  \begin{align}
    \label{eq:contain}
    \M^* \cap L_{\epsilon_1,\epsilon_2} \subset L^\M_{\epsilon_1,\epsilon_2}.
  \end{align}
  Consider some $X\in \M^* \cap L_{\epsilon_1,\epsilon_2}$, and let $s_1,s_2$ be non-parallel saddle connections with $|s_i|\le \epsilon_i$.  To prove the claim it suffices to show that $s_1,s_2$ are $\M$-independent.  Otherwise, since $s_1$ and $s_2$ both have non-zero length, we would have $s_1 -\alpha s_2=0$ for some $\alpha \in \C$, where we think of both sides as linear functionals on $T_X\M\subset H^1(X,\Sigma;\C)$.  Now since $\M$ is locally cut out by equations with \emph{real} coefficients, the annihilator $\operatorname{Ann}(T_X\M)$ in $H_1(X,\Sigma;\C)$ is spanned by real homology classes. 
  Since $s_1-\alpha s_2$ is in $\operatorname{Ann}(T_X\M) $, it then follows that $s_1-\operatorname{Re}(\alpha) s_2 \in \operatorname{Ann}(T_X\M)$.
  Since neither saddle connection can have zero holonomy vector, $\operatorname{Re}(\alpha)\ne 0$, and then we see that $s_1$ and $s_2$ must be parallel on $X$, contradiction.  Hence $s_1,s_2$ are $\M$-independent, and so we get \eqref{eq:contain}.
  
  Using that $\mu^1_{\M}$ is supported on $\M^*$, the inclusion \eqref{eq:contain}, and Theorem~\ref{thm:regularity}, we get
  $$\mu(L_{\epsilon_1,\epsilon_2}) = \mu^1_\M(\M^*\cap L_{\epsilon_1,\epsilon_2})\le \mu^1_\M(L_{\epsilon_1,\epsilon_2}^\M) = O(\epsilon_1^2 \epsilon_2^2).$$
\end{proof}

\subsection{Strata of $k$-differentials}

A $k$-differential on a Riemann surface, for $k$ a positive integer, is a form that can locally be written as $f(z)(dz)^k$, where $f$ is holomorphic; i.e. it is a section of the $k$-th power of the canonical line bundle.  

\begin{thm}[Finiteness of volume]
  \label{thm:kdiff}
  Let $\mu$ be the canonical linear measure whose support is a whole stratum $\K$ of $k$-differentials.  Then $\mu^1$ (supported on the unit area locus and defined via coning, as above) is finite: $$\mu^1(\K^1) <\infty.$$
\end{thm}

This was proven recently by Nguyen \cite{nguyen2019}, who also defines the measures $\mu$ above using the symplectic form on the absolute homology.  The theorem also follows easily from the results in this paper as we now show.  

\begin{proof}
For any holomorphic $k$-differential $(M,q)$, there is a canonical construction that produces a pair $(\hat M,\omega)$ with a finite cover $f:\hat M \to M$ such that $\omega$ is a holomorphic 1-form on $\hat M$ satisfying $\omega^k=f^*q$.  Applying this construction to all elements of $\K$ embeds $\K$ as some locus $\hat \K$ in some stratum $\H$ of holomorphic 1-forms.  This locus is locally cut out by linear equations in period coordinates, where the equations have complex coefficients (in fact, the equations can be chosen to lie in $\mathbb{Q}(\zeta_k)$, where $\zeta_k$ is a primitive $k$-th root of unity).   The measure $\mu$ becomes a measure supported on $\hat \K$ that is linear in period coordinates.  And $\hat \K$ is algebraic since it can be defined in terms of algebraic conditions.  So  according to Remark \ref{rmk:complex}, we can apply Theorem \ref{thm:regularity} to $\hat \K$, and in particular we get $\mu^1(\K^1) < \infty$.  
\end{proof}

\subsection{Heuristic for Theorem~\ref{thm:regularity}}

Here is the reason that the bound in Theorem \ref{thm:regularity} is natural.  The affine measure $\mu_\M$ is defined as Lebesgue measure on a linear subspace of period coordinates.  Locally, we should be able to pick a basis for period coordinates that contains the short saddle connections.  The period of a saddle connection of length at most $\epsilon_i$ lies in a ball of area $\epsilon_i^2$ in $\C$.  So each independent $\epsilon_i$-short saddle connection should lead to an $\epsilon_i^2$ factor in the volume.

When we complete our set of short saddle connections to a full basis for functionals on the subspace given by $\M$, each new period will be either (i) bounded, or (ii) a cross curve of a cylinder of small circumference.  Dealing with type (i) periods is easy, since they just increase the implicit constant in the $O(\cdot)$.  Type (ii) periods are potentially unbounded, but they still lie in a region of bounded area once the circumference curve of the cylinder has been fixed.

There are two main issues in making the above heuristic into a proof.
\begin{enumerate}
\item We need to find a finite system of period coordinate charts to do the ``local'' computation; since $\H$ (and $\M$) are non-compact, some of these charts must necessarily also be non-compact.  Our method is to use the \emph{Moduli space of multi-scale differentials}, a nice compactification $\barPH$ of the (complex) projectivization $\P\H$ recently introduced by Bainbridge-Chen-Gendron-Grushevsky-M\"oller \cite{bcggm2019}.  We define our system of period coordinate charts in terms of neighborhoods of boundary points.  
\item We have to choose a nice basis of relative homology adapted to the short saddle connections.  
\end{enumerate}

\subsection{Outline of the proof of Theorem~\ref{thm:regularity}}
\begin{itemize}
\item In Section~\ref{sec:lms}, we give background on the multi-scale compactification $\barPH$, which is a compactification of the projectivized stratum $\P\H$; it plays a key role in the proof.   This space admits a nice system of analytic coordinates defined in terms of plumbing and scaling.  
\item In Section~\ref{sec:semianalytic-pieces}, for each boundary point $\bar X \in \barPH$, we find a finite system of semianalytic period coordinate charts that cover a neighborhood of $\bar X$ (minus the boundary).  They are cut out by conditions on the analytic coordinates.  A major challenge is showing that the period map on each such set is injective (which is needed for it to be chart).
\item In Section \ref{sec:intersect}, we use Filip's theorem on algebraicity of affine invariant manifolds to show that our $\M$ intersects each chart above in finitely many components.  This implies that in each chart, $\M$ is a union of finitely many linear subspaces.  
\item In Section~\ref{sec:vol-comp}, we prove the desired volume bound for $\M$ restricted to each chart.  This involves several lemmas that give estimates on periods of surfaces in each such chart.  
\item Finally, in Section~\ref{sec:proof}, we deduce the global volume bound for $\M$ from the bound in each chart using the compactness of $\barPH$ (and the fact that the volume bound is trivial in the compact part of $\H$).   
\end{itemize}

\subsection{Acknowledgements}

I am very grateful to Alex Wright for suggesting the problem and for helpful conversations and guidance.  I also would like to thank Matt Bainbridge and Sam Grushevsky for patiently explaining their work to me.   I thank Frederik Benirschke, Dawei Chen, Martin M\"oller, and Jenya Sapir for useful discussions.  

Finally, I gratefully acknowledge the support of the Fields Institute during the Thematic Program on Teichm\"uller Theory and its Connections to Geometry, Topology and Dynamics, where some of this work was done.

\section{Multi-scale compactification}
\label{sec:lms}

We will use the multi-scale compactification $\barPH$ introduced in \cite{bcggm2019} to define a finite system of period coordinate charts for $\H$.  The material in this section is a review of certain parts of \cite{bcggm2019}, though our notation is somewhat different, since certain technical aspects of their constructions are not relevant for us.    

Several other compactifications of strata have been considered, namely the ``What You See Is What You Get'' compactification studied by Mirzakhani-Wright \cite{mw2017}, and the Incidence Variety Compactification \cite{bcggm2018}.  The multi-scale compactification is best suited to our purposes since:
\begin{enumerate}[(i)]
\item it is ``large'' in the sense that there is a lot of data associated to each boundary point, so the types of surfaces that lie near a particular boundary point are more restricted, and 
\item it has nice analytic properties - in particular it is a smooth complex orbifold with explicit analytic coordinates defined in terms of plumbing and scaling.  
\end{enumerate}

The boundary points of the multi-scale compactification are nodal Riemann surfaces, where each component is equipped with a meromorphic differential, together with additional combinatorial data.  Components with poles of order $2$ or greater represent subsurfaces that are approaching zero size; the differential on such a component comes from appropriately rescaling the differential so that it converges to something not identically zero.  Poles of order $-1$ come in pairs and each pair represents a degenerating cylinder, i.e. one whose circumference is much smaller than any curve crossing it.

Because the compactification remembers the shape of subsurfaces that are vanishing, one can choose a basis of homology consistently across all the surfaces in a small neighborhood of a boundary point (actually, the neighborhood must be divided into finitely many pieces first).  This is done in Section~\ref{sec:vol-comp}, and it is a crucial part of the volume computation.  

\subsection{Definition of multi-scale compactification $\barPH$}

We recall the definition and properties of the multi-scale compactification.  The complete definition is quite complicated, so we refer the reader to \cite{bcggm2019} for full details.

Let $\H(\mu)$, $\mu=(m_1,\ldots,m_n)$, be a stratum of holomorphic 1-forms with $n$ labeled zeroes whose vanishing orders given by the $m_i$.
\footnote{Theorem \ref{thm:regularity} also holds for strata in which the zeros are \emph{unlabeled}.  These strata are quotients of the labeled strata by an appropriate finite subgroup of the symmetric group, so they share the same qualitative properties.}
Since $\mu$ will be fixed, we will use the shorthand $\H = \H(\mu)$.   The space $\barPH$ will be formed by attaching to $\P\H$ certain \emph{multi-scale differentials} $\bar X$ which consist of stable Riemann surfaces with meromorphic differentials on the components, plus some combinatorial data, up to a certain equivalence relation.

Our primary interest is in $\H$, but the natural object to compactify in a complex analytic way is the complex projectivization $\P \H$.  So we work with this space, and then translate statements back to $\H$.  

Let $\barM_{g,n}$ denote the Deligne-Mumford compactification of the moduli space of genus $g$ curves with $n$ marked points.  An element $M\in \barM_{g,n}$ is a stable, nodal Riemann surface.  Associated to $M$ is the \emph{dual graph} $\Gamma_M$: the  vertices $v$ correspond to the components $M_v$ of $M$, and for each node of the Riemann surface there is an edge connecting the vertices corresponding to the two components (possibly the same) joined by the node.

\begin{figure}[]
\begin{center}
  \includegraphics[scale=0.73]{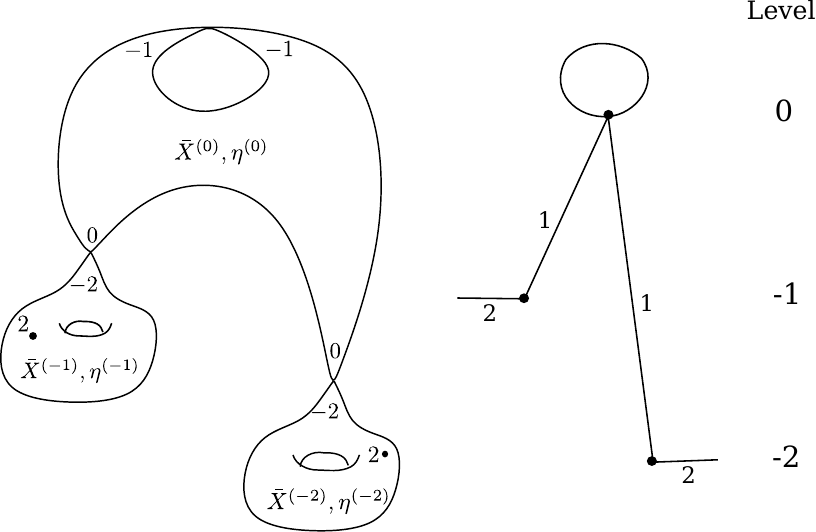}
  \caption{Left - A stable Riemann surface with a twisted differential.  This comes from a degeneration of surfaces in the stratum $\H(2,2)$.  The numbers at the nodes indicate the order of the differential, while the two dots correspond to the two zeros that surfaces in $\H(2,2)$ have.  Right - The corresponding enhanced level graph.}
  \label{fig:surfaceH22}
\end{center}
\end{figure}

A multi-scale differential will come with the following additional structure on the dual graph; it is related to the differentials on the components:
\begin{enumerate}
\item \emph{Half-edges}, each having a single vertex as endpoint.  There will be one of these for each of the $n$ integers in the type $\mu$ that specifies the stratum $\H=\H(\mu)$.
  The purpose of these is to record on which components the zeros of surfaces in $\H$ end up in the limit.   Each is labeled by a positive integer which is the order of the corresponding zero.  
\item A (surjective) \emph{level function}
  $$\ell: V(\Gamma) \to \underline{N} := \{0,-1,\ldots,-N\},$$
  which determines a weak order on the vertices.
  This ordering will determine which of two given components on smooth translation surfaces near the multi-scale differential is larger.
  Edges that join components at the same level will be called \emph{horizontal edges}, while those that join components at different levels will be called \emph{vertical edges}.  
\item An assignment of a positive integer $b_e$ to each vertical edge.  The cone angle around the node will be the same on either side of the node and will equal $2\pi b_e$.  
\end{enumerate}

A dual graph equipped with all this structure will be called an \emph{enhanced level graph of type $\mu$}.  

\paragraph*{Twisted differentials.}
A \emph{twisted differential} on $M$ \emph{compatible} with the enhanced level graph $\Gamma$ is a collection of meromorphic differentials $\eta= \{\eta_v\}$, one for each irreducible component $M_v$ of $M$, that is consistent with $\Gamma$ in a sense suggested above (see \cite{bcggm2019} for full definition).  The term ``twisted'' is meant to remind us that these are meromorphic differentials, and so are related to sections of twists of the canonical bundle.  Several consistency conditions ensure that the limit objects can be smoothed to surfaces in $\P\H$.  One condition is that at a vertical node, the cone angle of the differential above equals the cone angle of the differential below.  This is equivalent to the orders of the two differentials on either sides summing to $-2$.  A more subtle part of the definition of compatibility is the Global Residue Condition, which is not always forced by the residue theorem.

We will call the union of components at level $i$, i.e. those corresponding to the vertices of $\ell^{-1}(i)$, the \emph{level subsurface} at level $i$ and denote it $\bar X^{(i)}$.  We will denote the restriction of $\eta$ to $\bar X^{(i)}$ by $\eta^{(i)}$.   At times, by $\bar X^{(i)}$ we will actually mean the pair $(\bar X^{(i)}, \eta^{(i)})$.  

\paragraph*{Prong matching.}
The final piece of data is a \emph{prong-matching}.  This is only needed when there is at least one pole of order $k\le -3$.  At the other side of the corresponding node, there is a zero of order $-2-k\ge 1$.   The cone angle around each is $2\pi (-k-1)$.  At each such point, there is more than one ``prong'', i.e. a choice of horizontal direction.  In order to smooth the limit point, we must know how the prongs from the two differentials will be identified - this is what the prong-matching records.

\begin{defn}[Multi-scale differential]
  A \emph{multi-scale differential} $\bar X$ of type $\mu$ is the data $(M, \Gamma, \eta, \sigma)$ where $M$ is a stable Riemann surface, $\Gamma$ is an enhanced level structure of type $\mu$ on the dual graph of $M$, $\eta$ is a twisted differential on $M$ compatible with $\Gamma$, and $\sigma$ is a prong-matching.  
\end{defn}

\paragraph*{Level rotation torus.}
Two multi-scale differentials are considered equivalent if they differ by the action of the \emph{level rotation torus}.  This action comes from simultaneously rescaling all the $\eta_v$ at the same level by some element of $\C^*$.  However, moving around a circle in $\C^*$ may change the prong-matching at nodes at that level.  This subtlety means that instead of directly taking an action of copies of $\C^*$, we instead consider the universal cover $\C$ of each copy.   Keeping track of the prong-matching is needed to get a smooth orbifold, but its presence does not play a major role in our proof.   (Note that we include rescaling of the top level in the level rotation torus, since we wish to directly define a compactification of the projectivized stratum; this is one of the spaces \cite{bcggm2019} work with, but they also consider a version that does not involve rescaling of the top level.)

\begin{defn}
  The (projectivized) \emph{Moduli space of multi-scale differentials} $\barPH(\mu)$ is the set of multi-scale differentials of type $\mu$, modulo the equivalence relation described above.  
\end{defn}

The following is part of \cite [Theorem 1.2]{bcggm2019}.  
\begin{thm}
  The set $\barPH$ has a natural structure of a compact complex orbifold.
  \label{thm:compact}
\end{thm}

\subsection{Coordinates for $\barPH$}
\label{sec:coords}
From the way it is constructed in \cite{bcggm2019}, $\barPH$ admits a nice system of complex-analytic coordinates (in the orbifold sense).  We now describe these near a boundary point $\bar X$.  There are two types of coordinates:
\begin{enumerate}[(i)]
  \item \emph{Moduli parameters} $s_i$.  Changing these moves the surface parallel to the boundary stratum.
\item \emph{Smoothing parameters}. Deforming a smoothing parameter away from $0$ will smooth certain nodes of $\bar X$. These parameters fall into two categories:
  \begin{enumerate}
  \item a \emph{scaling parameter} $t_{i}$ for each level $i$ below the top level (which is level $0$), and 
  \item a \emph{horizontal node parameter} $t$ for each horizontal node.  
      \end{enumerate}
\end{enumerate}

We will describe how to construct a surface $X$ specified by moduli and smoothing parameters.  
There are several choices that need to be made to define the coordinates; we will discuss these along the way, but they should really be made at the beginning of the construction.  We will describe plumbing applied to single surface, but to get a complex orbifold, it is crucial that this construction can be done in a holomorphically varying way for holomorphic families.  This is subtle; it is dealt with carefully in \cite{bcggm2019}.

When all the smoothing parameters are non-zero, this will produce a particular \emph{distinguished}  surface $X\in \H$, not just a projective equivalence class in $\P\H$.  In some of the lemmas we will use this as a reference surface to discuss sizes.  

\subsubsection{Moduli parameters}
\label{sec:moduli}

We first describe the role of the moduli parameters $s_i$.  Note that $\bar X$ lies in a \emph{boundary stratum} consisting of twisted differentials in $\barPH$ that share the same combinatorial data $\Gamma,\sigma$ (the enhanced level graph and prong-matching). 
This boundary stratum can be identified with a suborbifold of a certain stratum of meromorphic differentials cut out by certain conditions on residues coming from the Global Residue Condition and the requirement that residues match at pairs of simple poles.  
This space naturally has a complex structure, and thus we can take a system of complex-analytic coordinates for (the projectivization of) the boundary stratum near $\bar X$.

A good choice is to use period coordinates.  Strata of meromorphic differentials admit period coordinates given by the cohomology group $H^1(Y\backslash P, Z;\C)$ of the surface $Y$ minus the poles $P$, relative to the set $Z$ of zeros and marked points.  Note that for us $Y$ will consist of the normalization of the possibly nodal Riemann surface $M$, so it will often have many components.  Since the residues are given by periods of cycles that encircle punctures, all the residue conditions are linear in period coordinates.  To account for the quotient by the action of the level rotation torus, for each level we take the projectivization of the relevant subspace of period coordinates for the meromorphic stratum corresponding to that level.

To get actual local coordinates, we will normalize by taking a fixed relative cycle $\gamma$ at each level to have period $1$.  The existence of such a cycle that has non-zero period at $\bar X$ is guaranteed by the stability condition (otherwise, integrating the $1$-form would give a branched covering to $\hat \C$, but the number of preimages of $\infty$, counted with multiplicity, would be too small).
Normalizing in this way we get local coordinates for the boundary strata which will be the moduli parameters $s_i$.  We will sometimes translate to arrange that all $s_i$ are $0$ at $\bar X$.  

\subsubsection{Smoothing parameters}
\label{sec:smoothing}

We now move on to the smoothing parameters, which are more complicated. 

To start off, we must pick a particular multi-scale differential $(M,\Gamma,\eta,\sigma)$ in the equivalence class of $\bar X$ (recall that $\barPH$ is the set of multi-scale differentials modulo the action of the level rotation torus).  In fact, we must pick such a representative for every surface in a small neighborhood of the boundary stratum of $\bar X$, and further, we do this in such a way that the differentials $\eta_v$ that compose $\eta$ vary holomorphically.   In other words, we are choosing a distinguished holomorphic section of the quotient (by level rotation torus) map locally for each level.  
To get this section, we use the normalization used above in Section \ref{sec:moduli}, defined by choosing a particular relative cycle at each level to have period $1$.  This gives surfaces $(M,\eta)$ compatible with $\Gamma$.  We then pick an appropriate prong matching $\sigma$ so that each $(M,\Gamma,\eta,\sigma)$ lies in the equivalence class of the desired boundary point in $\barPH$.

\paragraph{Standard coordinates near a node.}
\label{par:standard}
To perform plumbing, we find \emph{standard coordinates} on $\bar X$ near its nodes.  For each node we can find complex-analytic coordinates on the Riemann surface near that node such that the differential has the form
\begin{align}
      \label{eq:standard}
  \begin{cases}
    z^k dz & \text{if } k\ge 0\\
    \frac{r}{z} dz & \text{if } k=-1 \\
    \left(z^k + \frac{r}{z} \right) dz & \text{if } k\le -2,
  \end{cases}
\end{align}
where $k$ is the order of vanishing of the differential at the node. Here $r$ is the residue at the pole in the case $k\le -1$.  
The existence of these coordinates was proved by Strebel \cite{strebel1984} (see also \cite[Theorem 4.1]{bcggm2019}, where they prove that standard coordinates are essentially unique and vary holomorphically as we vary $\bar X$).  When $k\ge 0$, the coordinates are unique up to multiplication by a $(k+1)$-st root unity.  When $k=-1$, we can rescale by any complex number.  When $k\le -2$, the coordinates are unique after specifying a fixed base point $p$ near the node to have some fixed coordinate (which we will often take to be $1$).

\paragraph{Scaling parameters.}
\label{sec:plumb}
We will first discuss the role of the scaling parameters $t_i$.  Such a parameter determines both the size of the level $i$ subsurface relative to the level $i+1$ subsurface and the way that the level $i$ subsurface is plumbed into the higher levels.

\begin{figure}[]
\begin{center}
  \includegraphics[scale=0.8]{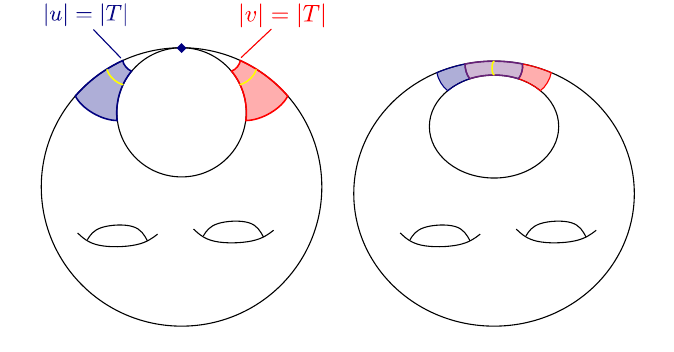}
  \caption{Classical plumbing applied to a Riemann surface with a single node. (Modification of a figure created by Jenya Sapir; used with permission.)}
  \label{fig:classical_plumbing}
\end{center}
\end{figure}

\subparagraph{Poles of order $2$, no residues.} We will first describe the simpler case in which all poles of $\bar X$ are of order $2$, and all residues at poles are zero (at $\bar X$ as well as at all nearby surfaces).  Consider two level subsurfaces $\bar X^{(i)}$ and $\bar X^{(j)}$ of $\bar X$, with $i>j$, and suppose that $u$ and $v$ are  standard coordinates near two points of $\bar X^{(i)}, \bar X^{(j)}$, respectively, that are joined together in a node (we will assume that the discs $\{u: |u|\le 1\}$ and $\{v: |v|\le 1\}$ are contained in the coordinate charts; otherwise we should remove a somewhat larger disc in the next paragraph before doing the plumbing).  The orders of vanishing at the two points that get joined will sum to $-2$.  Let
$$T=t_{i-1} t_{i-2}\cdots t_{j}.$$
To smooth the node, at the level of Riemann surfaces, we perform \emph{classical plumbing}, as in Figure \ref{fig:classical_plumbing}.  That is, we first remove the small discs $\{u: |u|\le |T|\}$, $\{v: |v| \le |T|\}$ from $\bar X^{(i)}, \bar X^{(j)}$, respectively.  For remaining points in the sets parameterized by $u,v$, we glue $u$ to $v$ whenever
$$uv=T.$$

For the differential on $\bar X^{(i)}$, we take
$$t_{-1}t_{-2}\cdots t_i \cdot \eta^{(i)},$$
and on $\bar X^{(j)}$, we take
$$t_{-1}t_{-2}\cdots t_j \cdot \eta^{(j)}.$$

Next we choose standard coordinates \eqref{eq:standard} near the nodes (actually, it is crucial that this choice be made at the very beginning of the construction; see discussion at beginning of Section \ref{sec:coords}).  We need these coordinates to vary holomorphically as we vary the boundary point by a small amount.  We get this, provided that we choose in a holomorphic manner the fixed base point $p$ on the nearby boundary surfaces used to normalize the coordinates.  %
See \cite[Section 10]{bcggm2019} for details.

In these coordinates, on the upper component $\bar X^{(i)}$ the differential has the form 
$$t_{-1}t_{-2}\cdots t_i \cdot du,$$
and on the lower component $\bar X^{(j)}$ it has the form
$$t_{-1}t_{-2}\cdots t_j \cdot \frac{dv}{v^2}.$$
The following calculation shows that the upper and lower differentials agree at points that are glued together:
\begin{align*}
  t_{-1}t_{-2}\cdots t_i  \cdot du &= t_{-1}t_{-2}\cdots t_i \cdot  d(T/v)  \\
                                   &=  -t_{-1}t_{-2}\cdots t_i \cdot t_{i-1} t_{i-2}\cdots t_{j} \cdot \frac{dv}{v^2}.
\end{align*}
(Because of the factor of $-1$, on the lower component we modify the coordinates so that the differential is the negative of the standard form (\ref{eq:standard}) above).  Hence in this case we get a well-defined differential on the plumbed surface.

\subparagraph{Poles of order 2 or higher, no residues.} We now generalize to higher order poles, still assuming that all residues are zero (at $\bar X$ as well as at all nearby surfaces).  For each level $i$, we assign an integer $a_i$, taken to be the least common multiple of the $b_k$ associated to all nodes that join a component at level greater than $i$ to a component at level $i$ or smaller, where $2\pi b_k$ is the cone angle around the node.   

For the differential on $\bar X^{(i)}$, we take
$$t_{-1}^{a_{-1}} \cdots t_i^{a_i} \cdot \eta^{(i)},$$
and on $\bar X^{(j)}$, we take
$$t_{-1}^{a_{-1}}\cdots t_j^{a_j} \cdot \eta^{(j)}.$$

At a node with cone angle $2\pi b$, we choose holomorphically varying standard coordinates \eqref{eq:standard} (this choice should actually made at the very beginning of the construction).  %
In these standard coordinates near the node, the differential on the upper component $\bar X^{(i)}$ has the form
$$t_{-1}^{a_{-1}}\cdots t_i^{a_i} \cdot u^{b-1} du,$$
and near the node on the lower component $\bar X^{(j)}$, the differential has the form
$$t_{-1}^{a_{-1}}\cdots t_j^{a_j} \cdot \frac{dv}{v^{b+1}}.$$
We perform plumbing as in the previous case, gluing together points $u,v$ whenever $uv=T$, but now we use
$$T=t_{i-1}^{a_{i-1}/b}\cdots t_{j}^{a_{j}/b}.$$

A short calculation as in the previous case shows that the upper and lower differentials agree (up to factor $-1$) at points that are glued together.

\subparagraph{Modification differentials to account for residues.}
In the case when the residue $r$ (at a node on the lower level subsurface $\bar X^{(j)}$) is non-zero, we have to work to create a residue on the upper level subsurface $\bar X^{(i)}$ so that the differentials will match when we do the plumbing.  The solution is to create a \emph{modification differential} $\xi$ on the underlying Riemann surface of $\bar X^{(i)}$ that has a simple pole at the node with residue $r$.  The existence of holomorphically varying modification differentials with the required properties is furnished by \cite[Proposition 9.3]{bcggm2019} (or \cite[Lemma 4.6]{bcggm2018}).  However, the modification differential will not be uniquely specified, so we must make a choice of one at the beginning of the construction of coordinates.   In our discussion here, we will focus on the modification differential that comes from residues on $\bar X^{(j)}$.  In general there will be several modification differentials coming from residues at different levels and these should all be added to the higher level differential.

As in the case with no residues, on a neighborhood of each node, we can choose a coordinate $u$ such that the differential on $\bar X^{(i)}$ has the form $u^{b-1}du$.  Adding the scaled modification differential gives us the differential
\begin{align}
  u^{b-1}du+t_{i-1}^{a_{i-1}}\cdots t_{j}^{a_{j}}\xi.
  \label{eq:mod-diff}
\end{align}
Now \cite[Theorem 4.2 and Theorem 4.3]{bcggm2019} (or \cite[Theorem 4.3]{bcggm2018}) gives that by performing a change of coordinates (depending holomorphically on the $t_k$, as well as on the moduli parameters) on a fixed \emph{annulus} near $0$, we can assume that this differential has the form
\begin{align}
  \label{eq:stand_form}
\left(u^{b-1}+t_{i-1}^{a_{i-1}}\cdots t_{j}^{a_{j}} \frac{r}{u}\right) du.
\end{align}
Because the order of vanishing of \eqref{eq:mod-diff} is not constant (the simple pole disappears when all $t_k$ are zero), it is not possible to extend these coordinates to a full neighborhood of $0$.  Due to this issue, we do not define the plumbing directly on $\bar X^{(i)}$; instead we define a local family of \emph{plumbing fixtures}, and then glue the ends of these to the top and bottom pieces to realize the plumbing.   

Consider a disc $U$ with a differential on it given by the formula \eqref{eq:stand_form}.  We then rescale this differential on $U$ so that it has the form
$$t_{-1}^{a_{-1}}\cdots t_i^{a_i} \left(u^{b-1}+t_{i-1}^{a_{i-1}}\cdots t_{j}^{a_{j}} \frac{r}{u}\right) du.$$

Near the node on $\bar X^{(j)}$, in standard coordinates \eqref{eq:standard}, and after rescaling, the differential has the form
$$t_{-1}^{a_{-1}}\cdots t_j^{a_j}\left(\frac{1}{v^{b+1}} + \frac{r}{v}\right) dv.$$
We define another disc $V$ with a differential given by the above expression.

We then glue together those points $u\in U$ and $v\in V$ with $uv=T=t_{i-1}^{a_{i-1}/b}\cdots t_{j}^{a_{j}/b}$, as in the case without residues.  A short calculation similar to the one in the previous case shows that the two differentials above match up at the points glued together.  This is the local plumbing fixture (it is topologically an annulus).

By construction, on the plumbing fixture, one annular end matches the differential on the fixed annulus on $\bar X^{(i)}$, and the other annular end matches the differential on an annulus on $\bar X^{(j)}$.   Hence, after removing the appropriate discs on $\bar X^{(i)}$ and $\bar X^{(j)}$, and then gluing in the plumbing fixture, we get a plumbed surface with differential.

By the \emph{plumbing region} we will mean the part of the smooth surface constructed above that corresponds to the plumbing fixture.

\subparagraph{Merging zeros.}
The modification differential solves the residue matching issue but introduces a new problem: it may split up some of the higher order zeros of the original differential.  So afterwards, we must \emph{merge the zeros} that have been separated.  This amounts to cutting out a small disc containing the zeros, and then gluing in another disk equipped with a form that has a single zero.  For details see \cite[Theorem 4.2]{bcggm2019} (or \cite[Lemma 4.7]{bcggm2018}). 

\paragraph{Horizontal node (degenerating cylinder) parameters.}
The case when we have a node joining two points that lie on the same level subsurface $\bar X^{(i)}$ is similar to the higher order pole case, but we don't have to worry about modification differentials.  On either side of the node are simple poles.  On nearby surfaces, this node will become a degenerating cylinder.  We find standard coordinates $u,v$ given by \eqref{eq:standard} on each side of the node.  The residues at the two simple poles are forced to be negatives of one another, by the definition of multi-scale differential.   So in these coordinates, on either side of the two nodes the differential has the form
$$t_{-1}^{a_{-1}}\cdots t_i^{a_i} \cdot \frac{r}{u} du,$$
and
$$t_{-1}^{a_{-1}}\cdots t_i^{a_i} \cdot \frac{-r}{v}dv,$$
respectively.  
On the level of Riemann surfaces, we do classical plumbing as above, gluing $u$ to $v$ whenever $uv=t$, where $t$ is the horizontal node parameter.  A short calculation very similar to the ones done in the case of higher order poles gives that the two differentials match up at points that are glued together.  Hence we get a well-defined differential on the plumbed surface.

\section{Semianalytic period coordinate charts}
\label{sec:semianalytic-pieces}

In this section we will work with a fixed boundary point $\bar X \in \barPH$.

We let
$$p:\H \to \barPH$$
be the natural projectivization map composed with the inclusion into the compactification (taking this as the target allows us to pull back neighborhoods of a boundary point). For a subset $S\subset \H$, we define $\P S := p(S)$.

Our goal is to define a finite system of semianalytic period coordinate charts $\{V_k\}$ in $\H$, each invariant under $\C^*$-scaling, and such that $\{\P V_k\}$ cover $U\cap \P\H$ for some neighborhood $U$ of $\bar X$ in $\barPH$.  In each such period coordinate chart, an affine invariant manifold $\M$ will be given by a finite union of linear subspaces.  Hence, equipped with this system of finitely many period coordinate charts, we can understand $\M$ in terms of finitely many linear spaces.    

Each $\P V_k$ will be ``multi-sector'' i.e. a product of small sectors with respect to analytic coordinates near $\bar X$.  

\subsection{Defining the period coordinate charts $V_k$}
\label{sec:vk}
Choose a system of analytic coordinates near the boundary point $\bar X$, as in Section~\ref{sec:coords}, where $t$ are the smoothing parameters, and $s$ are the moduli parameters (translated so that the value is $0$ at the $\bar X$).  Each $\P V_k$ (which depend on this choice of coordinates) will then be defined by conditions as follows:

\begin{enumerate}[(i)]
\item  \label{item:epss}  Restrict $0<|s|<\epsilon$ for each moduli parameter $s$, for $\epsilon>0$ chosen sufficiently small so that these conditions cut out an open set in the boundary stratum on which period coordinates are well-defined and injective.  This is possible since periods give local coordinates for any meromorphic stratum, and the boundary stratum is locally a linear subspace of a meromorphic stratum. (In the proof of Lemma \ref{lemma:period-chart}, we will use arguments for which $\epsilon$ will need to be chosen yet smaller.) 

\item \label{item:epst} Restrict $0<|t|<\epsilon$ for each smoothing parameter $t$.
  
\item \label{item:horiz}  If $t$ is a horizontal node parameter, then we consider a restriction of the form
  $$\arg t \in (\alpha, \alpha+\pi/4).   $$
  We choose finitely many $\alpha$ so that the union of the intervals above for these $\alpha$ cover the full circle of directions.  Each $\P V_k$ will correspond to an interval given by one choice of $\alpha$.   (We need such conditions to ensure that $V_k$ admits an injective period map; see Example \ref{ex:noninj}.)  

\item \label{item:vert}  Suppose $t_i$ is the scaling parameter for level $i$.  We have defined an associated integer $a_i$ in Section \ref{sec:plumb} related to higher order poles.  We can find connected interval conditions on the $\arg t_i$ that imply that $\arg t_i^{a_i}$  satisfies a condition of the form in  (\ref{item:horiz}) above.  With finitely many such interval conditions, we cover all possibilites for $t_i$.  Each $\P V_k$ will correspond to one of these interval conditions for each such $t_i$. 
\end{enumerate}

We then take $V_k := p^{-1}(\P V_k)$.  
\begin{rmk}
  In the above we were assuming that $\bar X$ was not an orbifold point.  If it is an orbifold point, we define sets in the local manifold cover as above, and then get the $\P V_k$ by pushing down to $\barPH$.  
\end{rmk}

\begin{example}
  Consider the example of a point $\bar X$ in $\overline{\P\H(3,1)}$ which has two levels connected by a single node, where the lower level piece is genus $1$, with one pole of order $3$ and one zero of order $3$, and the upper level piece is genus 2 with two zeroes of order $1$, one of which lies at the node.   See Figure \ref{fig:surfaceH13}.

  There are moduli parameters $s_1,\ldots,s_n$ that vary the bottom the top and bottom pieces within their strata.  There is one scaling parameter $t_{-1}$.  We have that $a_{-1}=2$, since there is only one pole, and its order is $3$ (recall that $a_{-1}$ is the least common multiple over all the poles of the associated integer $b_k$, which is defined so that the cone angle at that pole is $2\pi b_k$).  For $k=0,\ldots, 15$, define intervals:
  \begin{align}
    S_{2k+1} &= \left(k\frac{\pi}{8}, (k+1) \frac{\pi}{8}\right), \\
    S_{2k+2} &= \left(0.1 + k\frac{\pi}{8}, 0.1 + (k+1) \frac{\pi}{8}\right). 
  \end{align}
  The sets  $V_k$ in the definition can then be taken to be
  $$V_k = \{ X : |s_1|,\ldots, |s_j| < \epsilon, \text{ } |t_{-1}| < \epsilon, \text{ } \arg t_{-1} \in S_k \},$$
  for $k=1,\ldots,32$.  
\end{example}

\begin{figure}[h]
  \begin{center}
    \begin{subfigure}[b]{0.49\textwidth}
      \begin{center}
      \includegraphics[scale=0.6]{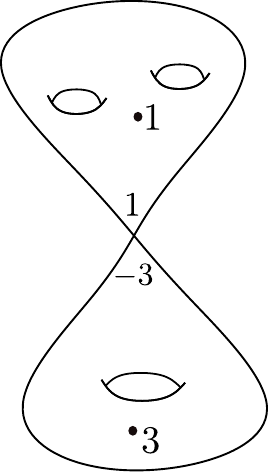}
      \caption{Surface in $\overline{\P\H(3,1)}$} 
      \label{fig:surfaceH13}
    \end{center}
    \end{subfigure}
    \begin{subfigure}[b]{0.49\textwidth}
      \begin{center}
      \includegraphics[scale=0.6]{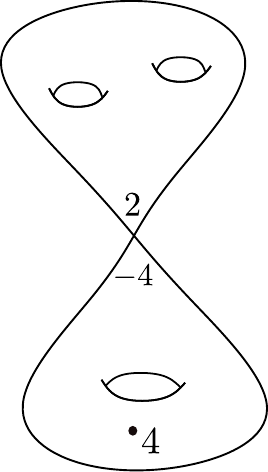}
      \caption{Surface in $\overline{\P\H(4)}$}
      \label{fig:surfaceH4}             
      \end{center}
    \end{subfigure}
    \caption{Boundary surfaces with higher order poles.}
  \end{center}
\end{figure}

\subsection{Properties of $V$}
In a real-analytic manifold, a \emph{semianalytic set} is a subset that is locally cut out by real-analytic equalities and inequalities.  In a real-analytic orbifold, a semianalytic set is a subset such that its preimage under some real-analytic orbifold chart is a semianalytic subset of $\mathbb{R}^n$.   See \cite{kankaanrinta2011}  for a complete definition.   

\begin{lemma}
  \label{lemma:semianalytic}
Each $\P V_k$ is simply connected and semianalytic, and
$$\bigcup_k \P V_k= U \cap \P\H,$$
where $U$ is the neighborhood of $\bar X$ given by restricting $|s|<\epsilon$ for each moduli parameter, and $|t|<\epsilon$ for each smoothing parameter.  
\end{lemma}

\begin{proof}
  Because of the conditions on angles, each $\P V_k$ is a product of convex subsets of $\C$, and is hence itself convex, hence contractible, and in particular simply connected. 

  The $\P V_k$ are semianalytic because each of the defining conditions can be expressed as an inequality on the real or imaginary part of a locally defined analytic coordinate function (the arg expression in  (\ref{item:horiz}) and (\ref{item:vert}) is not a well-defined analytic function, but equivalent conditions can be expressed in terms of ratios of real and imaginary parts).  

  The last claim about the union of the $\P V_k$ follows immediately from their definition.
  
\end{proof}

\begin{defn}[Period coordinate charts]
  \label{defn:period-chart}
  A connected, open subset $Q\subset \H$ is said to be a \emph{period coordinate chart} if it admits an injective map to $\C^n$ that is locally linear (the stratum locally has a linear structure coming from local period coordinates; near orbifold points one should work in an appropriate cover). 
\end{defn}

In particular, any sufficiently small neighborhood of a point $X\in \H$ will be a period coordinate chart.  However, it is \emph{not} necessarily true that for a small neighborhood $U$ of a boundary point $\bar X$ in $\barPH$, the set $p^{-1}(U)\cap \H$ is a period coordinate chart.  One issue is that the set might not be simply connected, in which case it is not possible to consistently choose a basis of relative homology to take periods of.

\begin{example}[Cautionary]
  \label{ex:noninj}
  Even on a simply connected subset of $p^{-1}(U)\cap\H$, the period coordinate map might not be injective.  To see this, consider a boundary point of the type shown in in Figure \ref{fig:surfaceH4}.  This has two levels, connected by a single node, with a pole of order $4$ below.  Let $X',X''$ be two surfaces near this boundary point whose coordinates are identical, except that for the scaling parameters $t'_{-1},t''_{-1}$, we have $t'_{-1}=\omega t''_{-1}$, where $\omega^3=1,\omega\ne 1$.   All periods of cycles coming from the top surface are the same for $X',X''$.  Since the bottom differentials are multiplied by $(t'_{-1})^3 = (t''_{-1})^3$, the periods of cycles from the bottom surface are also all equal.  Since we can form a basis out of such cycles that do not cross the nodes, we have that the all periods for $X',X''$ are equal.  On the other hand, the two surfaces must in fact be different, since they have different analytic coordinates (this also can be seen directly from the flat pictures).
\end{example}

Instead we use the sets $V$ defined in Section \ref{sec:vk}.

\begin{lemma}
  \label{lemma:period-chart}
  Each set $V =V_k \subset \H$ from Section \ref{sec:vk} is a period coordinate chart in the above sense, provided that the $\epsilon$ in the defintion of $V$ is sufficiently small (see Section \ref{sec:vk}, item (\ref{item:epss})).  
\end{lemma}

The proof of this lemma is somewhat involved; the necessary tools are developed in Section \ref{sec:tools}, and then proof is completed in Section \ref{sec:proof-lemma}.  

\subsection{Tools for proof of Lemma \ref{lemma:period-chart}}
\label{sec:tools}

Our first goal is to estimate periods of various cycles on surfaces in $V$.   Since $\P V$ is simply connected (Lemma \ref{lemma:semianalytic}), if we pick a relative cycle $\alpha$ representing an element of $H_1(X,\Sigma;\Z)$ on  some surface $X\in V$, we can consistently transport it to a cycle on all surfaces in $V$.  We will study how the period of this cycle depends on the analytic coordinates describing $X$.

To study the period of the cycle $\alpha$ on smooth surfaces, we will first introduce the \emph{perturbed period}, which comes from taking the period of the part of $\alpha$ that comes from the highest level subsurface that it interacts with.  This will be easier to understand than the full period, since it will give a well-defined holomorphic function in a full neighborhood of $\bar X$, including at the boundary.  
Our perturbed periods are a special case of the construction in \cite[Section 9]{bcggm2019}. 

\begin{defn}
  Given a relative cycle $\gamma$ on the level subsurface $\bar X^{(\ell)}$, the perturbed period is the holomorphic function
  $$\gamma^{\pert}: U \to \C,$$
  whose domain is a small neighborhood $U$ of $\bar X$ in $\barPH$, defined as follows.  In the plumbing construction of Section \ref{sec:coords}, after the various choices of coordinates, but before doing the plumbing, we truncate $\gamma$ by taking each representative curve that ends at a node and replacing it by a curve that ends at the point with coordinate $p$ (some fixed small complex number) in the coordinates chosen about that node.  Then $\gamma^{\pert}(X)$ is defined as the period over this truncated curve, before doing any rescaling.  See Figure \ref{fig:pert_periods}.  
\end{defn}

\begin{figure}[]
\begin{center}
  \includegraphics[scale=1]{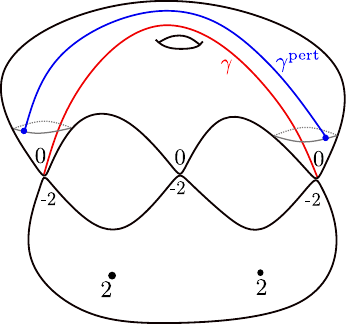}
  \caption{Construction of the perturbed period corresponding to a relative cycle $\gamma$ on a piece of the limit surface. }
  \label{fig:pert_periods}
\end{center}
\end{figure}

\paragraph*{Downward extension of cycles.}
Suppose $\gamma$ is a relative cycle on some level subsurface $\bar X^{(\ell)}$ of $\bar X$.  If we wish to consider this as a class on smooth surfaces in $V$, we can extend $\gamma$ ``downwards'' as follows.  Represent $\gamma$ by curves, and extend each curve going toward a vertical node downwards until a locally minimal component (in the level graph) is reached.  Such a component must have a zero of the differential, so a relative homology class can be produced.  See, for example, Figure \ref{fig:inj_periods}.  This produces cycles $\hat \gamma$ on surfaces in $X$.  Note that $\hat\gamma$ will not cross any degenerating cylinders (coming from horizontal nodes).  

In the next lemma, we estimate the period of $\gamma$ in terms of the perturbed period part, the part that crosses the plumbing region, and the lower level part.  We will focus on how this period depends on moduli parameters for $\bar X^{(\ell)}$ and the scaling parameter $t_{\ell-1}$.  The expression in the lemma has some non-explicit terms, but each of these either extends to a holomorphic function on a full neighborhood of $\bar X$ (and so the function, as well as its derivative, will enjoy good boundedness properties), or only depends on lower level moduli or scaling parameters.  
\begin{lemma}
  \label{lemma:period}
  Let $\hat \gamma$ be a cycle on smooth surfaces in $V$ obtained by starting with a relative homology class $\gamma$  on $\bar X^{(\ell)}$ and extending downwards (without crossing degenerating cylinders).  Let $t:=t_{\ell-1}^{a_{\ell-1}}$.  Then
  \begin{align*}
    \hat\gamma (X) = t_{-1}^{a_{-1}} \cdots t_{\ell}^{a_\ell} \left( \gamma ^{\pert}(X) + c + t  f(X)+  (t  \log t )g(X)  + t h(X) \right),
  \end{align*}
for $X\in V$, where:
\begin{itemize}
\item $c$ is a constant (which we can take small, if we choose the $p$ in the definition of perturbed periods to be small),
\item $\gamma ^{\pert}$ is the perturbed period coordinate,
\item $f,g$ are functions that extend to holomorphic functions in a full neighborhood of $\bar X$,
\item $h$ is a function of $X$, but it only depends on $X$ through moduli parameters at level lower than $\ell$, and scaling parameters $t_k$ with $k<\ell-1$.  Also, $h(X)$ is bounded above in absolute value as $X$ ranges over $V$ (we do not require that $h$ gives a holomorphic function on a full neighborhood of $\bar X$).  
\end{itemize}
\end{lemma}

\begin{proof}
  We begin by expressing $\hat\gamma$ as the union of three disjoint parts.  We cut $\hat \gamma$ at the base point $p$ on each lower component used in the plumbing construction.  We also need to cut at base points $p$ on the upper components; such base points can be introduced in the same way as for the lower components.   The three parts are then:
  \begin{enumerate}[(i)]
  \item the part that lies on $\bar X^{(\ell)}$ and away from the plumbing regions \label{item:ell}
  \item the part that lies in the plumbing regions \label{item:plumb} 
  \item the part that lies on level subsurfaces at level $j< \ell$ and away from the plumbing regions.  \label{item:lower}
  \end{enumerate}

  For (\ref{item:ell}), by the definition of the perturbed period $\gamma^{\pert}$ above, the contribution is equal to $t_{-1}^{a_{-1}} \cdots t_{\ell}^{a_\ell} \gamma ^{\pert}(X)$.
  
  For (\ref{item:plumb}), we will analyze each such plumbing region which $\hat\gamma$ crosses separately, and then we can add all the expressions together.   Each such region corresponds to a node joining level $i$ to level $j$, where $\ell\ge i > j$.  The contribution to $\hat\gamma(X)$ is given by
  $$t_{-1}^{a_{-1}} \cdots t_i^{a_i} \int_T^1 \left(u^{b-1} + T^b \frac{r}{u} \right) du,$$
  where $T=t_{i-1}^{a_{i-1}/b}\cdots t_{j}^{a_{j}/b}$, and $r$ (which comes from the modification differential) extends to a holomorphic function in a full neighborhood of $\bar X$.   We claim that the integral in the above equals an expression of the form
  $$ c+ t  f(X)+  (t  \log t )g(X)  + t h(X).$$
  
  To prove this, first note that $u^{b-1}$ term in the integrand gives a $c + t f(X)$ contribution, where $f$ is just some constant multiple of a power of $t$.  

  Now we split up into cases based on $i,j$ to analyze the contribution of the $T^b\frac{r}{u}$ term in the integrand.  
  \begin{itemize}

      \item   When $i=\ell, j=\ell-1$, we have $T = t_{\ell-1}^{a_{\ell-1}/b}$, and the contribution is 
    \begin{align*}
      \int_T^1 T^b \frac{r}{u} du = T^b r \left( \log 1 - \log T \right) = -  (r/b) t \log t,
    \end{align*}

    which becomes part of the $(t\log t) g(X)$ term.  %

  \item   When $i<\ell$, we do a computation similar to the one immediately above, and we ultimately get a contribution of the form $t f(X) + t h(X)$.  Here $h$ is a product of powers of various $t_k$ and a term of the form $\tau \log \tau$, where $\tau$ is also a product of powers of $t_k$.  Since $\tau\log\tau \to 0$ as $\tau\to 0$, this $h$ varies in a bounded away over $V$.

  \item   When $i=\ell$, $j<\ell-1$, we get a contribution of the form
    $$tf(X) + (t \log t)g(X)+th(X).$$
    To see this, we split up the $\log (t_{\ell-1}^{a_{\ell-1}/b} \cdots t_j^{a_j/b})$ in the expression we get by evaluating the integral into
    $$\log (t_{\ell-1}^{a_{\ell-1}/b})  + \log (t_{\ell-2}^{a_{\ell-2}/b} \cdots t_j^{a_j/b}).$$
    The first term behaves as in the $i=\ell, j=\ell-1$ case above, and the second behaves as in the $i<\ell$ case.  
  \end{itemize}

  For pieces of type (\ref{item:lower}), the period has the form $t_{-1}^{a_{-1}} \cdots t_{\ell}^{a_\ell} t_{\ell-1}^{a_{\ell-1}}h(X)$, where $h$ depends only parameters at level lower than $\ell$.  This $h$ varies in a bounded way as $X$ varies over $V$, since the moduli and scaling parameters on which it depends are bounded.  

  Combining the estimates for the pieces in (\ref{item:ell}), (\ref{item:plumb}),  (\ref{item:lower}) gives the desired estimate.  
  
\end{proof}

\subsection{Proof of Lemma \ref{lemma:period-chart}}
\label{sec:proof-lemma}

\begin{proof}[Proof of Lemma \ref{lemma:period-chart}]

  Since $\P V$ is simply connected (Lemma \ref{lemma:semianalytic}), we can choose a basis for relative homology at some surface in $V$ and then consistently transport it to a basis at all other surfaces in $V$.  Taking the periods of these classes gives the map $\phi: V \to \C^n$, which is clearly locally linear (recall Definition \ref{defn:period-chart}).

  It remains to show that $\phi$ is injective, which is rather technical, since it involves carefully estimating periods of cycles, including those that cross between different levels.  The reader is encouraged to first consider the proof in the case of a boundary point $\bar X$ of the form given by Figure \ref{fig:inj_periods}, where there are just two levels.   

  Suppose for the sake of contradiction that $\phi(X')=\phi(X'')$, with $X',X''\in V$ and $X'\ne X''$.  First suppose that the projective classes $[X'],[X'']$ are equal.  We have that $X' = c'X$ and $X'' = c''X$, where $X$ is the distinguished representative of the projective class produced by the plumbing construction (as discussed in Section \ref{sec:coords}), and $c'\ne c''$.  Thus $\phi(X') = c'\phi(X) \ne c''\phi(X) =\phi(X'')$, contradiction.   

  So we assume for the rest of the argument that $[X']\ne [X'']$, and we will prove that $[\phi(X')] \ne [\phi(X'')]$.  Since every relative homology class is a linear combination of the classes defining $\phi$, it is sufficient to find classes $\gamma_1,\gamma_2$ such that $\gamma_1(X')/\gamma_2(X') \ne \gamma_1(X'')/\gamma_2(X'')$.  To prove this statement about ratios, it clearly suffices to assume that $X',X''$ are the distinguished representatives produced in Section \ref{sec:lms}.  Since the projective classes are different, the analytic coordinates for $X',X''$ must be different.

  Our first task is to determine at which levels to look for these classes $\gamma_1,\gamma_2$, in terms of the analytic coordinates for $X',X''$.  To this end we will define the \emph{effective level} $\ell$ of each moduli, horizontal node, and scaling parameter (caution: this is somewhat different than the notion of level that will be used for subsurfaces and curves in Section \ref{sec:vol-comp}).  The motivation for the definition is that if two surfaces agree for all analytic coordinates at effective levels $\ell$ and lower, then ratios of periods that interact only with subsurfaces at level $\ell$ or lower should also be equal.  Here is the definition:  
  \begin{enumerate}[(i)]
  \item For a moduli parameter $s$ that deforms the level subsurface $\bar X^{(i)}$, the effective level is $i$.
  \item For a scaling parameter $t_i$, the effective level is $i+1$. (The reason we take the effective level higher than $i$ is that $t_i$ does not affect ratios of periods at levels $i$ and below.)  
      \item For a horizontal node parameter $t$, the effective level is $i+\frac12$, where $i$ is the level subsurface on which the pair of simple poles corresponding to the horizontal node lie.
  \end{enumerate}

  Consider all parameters that take different values for $X',X''$; among the effective levels of these, let $\ell$ be the lowest.  

  We now split into three cases based on the value of $\ell$.  In the first two cases (in which $\ell$ is an integer), we begin by choosing cycles $\gamma_1,\ldots,\gamma_n$ in the relative homology of $\bar X^{(\ell)}$ that give a basis for the space of functionals on the subspace of period coordinates cut out by the residue conditions, as in Section \ref{sec:moduli}.  See Figure \ref{fig:inj_periods}.  
  The periods of these give the moduli parameters corresponding to $\bar X^{(\ell)}$.  We take $\gamma_1$ to be the cycle that has period $1$ on all surfaces in the boundary stratum (recall from Section~\ref{sec:smoothing} that we are choosing the holomorphic section of the projectivization map by normalizing a particular cycle to have period $1$).    In all three cases we must prove that $[\phi(X')] \ne [\phi(X'')]$.

\medskip \noindent \underline{Case 1:} \emph{$\ell = -N$ (i.e. the surfaces differ at the lowest effective level)}

This case is straightforward because periods on the bottom level subsurface have a simple expression in terms of the analytic coordinates.

By the above definition of effective level, there must be some moduli parameter $s$ for the lowest level subsurface $\bar X^{(-N)}$ that has two different values $s'\ne s''$ for $X',X''$, respectively.
The parameter $s$ is a period of some relative homology class on $\bar X^{(-N)}$, which we can assume is $\gamma_2$.   Now since $\gamma_2$ is on the lowest level, it extends uniquely on all smooth surfaces near $\bar X$ to a relative homology class which does not cross any of the curves that are pinched at $\bar X$.

Now, the period of $\gamma_2$ differs between $X',X''$, while the period of $\gamma_1$ does not; hence the ratio of the two periods differs.  In fact we have the following explicit expressions for the periods on a smooth surface $X$ with parameter $s$:
$$\gamma_2(X) = t_{-1}^{a_{-1}} \cdots t_{\ell}^{a_\ell} \cdot s,$$
$$\gamma_1(X) = t_{-1}^{a_{-1}} \cdots t_{\ell}^{a_\ell}. $$
Hence
$$\frac{\gamma_2(X')}{\gamma_1(X')} = \frac{(t_{-1}')^{a_{-1}} \cdots (t_{\ell}')^{a_\ell} \cdot s'}{(t_{-1}')^{a_{-1}} \cdots (t_{\ell}')^{a_\ell}}= s' \ne s'' = \frac{(t_{-1}'')^{a_{-1}} \cdots (t_{\ell}'')^{a_\ell} \cdot s''}{(t_{-1}'')^{a_{-1}} \cdots (t_{\ell}'')^{a_\ell}} = \frac{\gamma_2(X'')}{\gamma_1(X'')},$$
and we are done since we have found a pair of periods with different ratios for $X',X''$.

\begin{figure}[]
\begin{center}
  \includegraphics[scale=1]{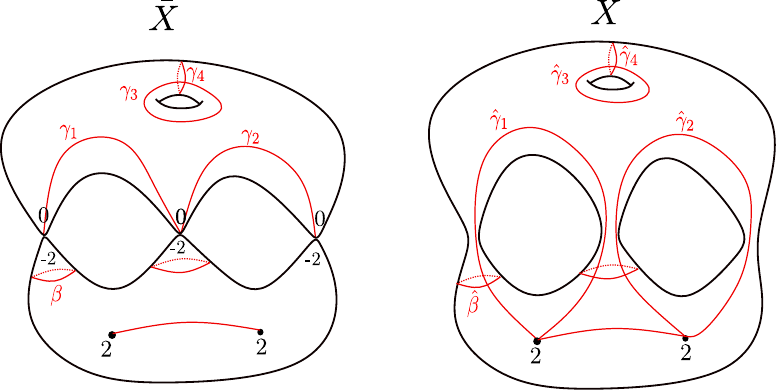}
  \caption{Extension of cycles on limit surface $\bar X$ to a smooth surface $X$ in the multi-sector $V$.  Case 2 applies if the moduli parameters for the bottom components of $X',X''$ are all the same. We are then forced to consider the cycles $\hat\gamma_1, \hat\gamma_2$ crossing between levels. }
  \label{fig:inj_periods}
\end{center}
\end{figure}

\medskip \noindent \underline{Case 2:} \emph{$\ell >-N$, and $\ell$ is an integer (i.e. the surfaces are the same at the bottom level, and the lowest effective level where they differ does not come from a degenerating cylinder)}

This case is involved since we have to understand how periods crossing between levels depend on the analytic coordinates (logarithmic terms appear) and how ratios of such periods behave.  

Recall that we have chosen $\gamma_1,\ldots,\gamma_n$ in the relative homology of $\bar X^{(\ell)}$ that give a basis for the space of functionals on the subspace of period coordinates cut out by the residue conditions; $\gamma_1$ is the cycle whose period is $1$ on all surfaces in the boundary stratum.  

We let $\beta$ be the relative homology class on $\bar X^{(\ell-1)}$ whose period is exactly $1$ at $\bar X$ and on nearby surfaces in the boundary stratum.  

Now the $\gamma_i$ classes, as well as the $\beta$ class, can be extended downwards on the smooth surfaces in $V$ to give relative homology classes $\hat \gamma_i$ and $\hat \beta$.  We require that these extended cycles do not cross any degenerating cylinders (coming from horizontal nodes).  This is the same procedure described in the setup before Lemma \ref{lemma:period}.

We will show that either $\hat\gamma_i(X')/\hat\gamma_1(X') \ne \hat\gamma_i(X'')/\hat\gamma_1(X'')$, for some $i$, or $\hat\beta(X')/\hat\gamma_1(X') \ne \hat\beta(X'')/\hat\gamma_1(X'')$.  To do this, we need to compute the periods of $\hat \gamma_i$ and $\hat\beta$ with respect to the analytic parameters $s_i$ for $i>1$ and $t:=t_{\ell-1}^{a_{\ell-1}}$.  Here $s_i$ is the moduli parameter corresponding to the period of $\gamma_i$ at the boundary.

By our assumption on effective levels, we know that 
$$ (s_2',\ldots,s_n',t')  \ne  (s_2'',\ldots,s_n'',t''),$$
(note that if $t' = t''$, then $t_{\ell-1}'= t_{\ell-1}''$, since our surfaces are in the multi-sector $V$ and so satisfy conditions given by (\ref{item:vert}) in Section \ref{sec:vk}).  

We will define a map $P$ on the subset of $\P V$ given by
\begin{align*}
  S:= \{ X \in \P V: &\text{ for every analytic coordinate } a \text{ with effective level} < \ell, \\
&\text{ the value of } a \text{ at } X \text{ agrees with the common value for } [X'],[X'']\}.
\end{align*}
Define $P$ by 
\begin{align*}
  P = (P_1,\ldots,P_n): S & \to \C^n \\
  (s_2,\ldots,s_n,t) & \mapsto \left( \frac{\hat\gamma_2(X)}{\hat\gamma_1(X)}, \ldots, \frac{\hat\gamma_n(X)}{\hat\gamma_1(X)}, \frac{\hat\beta(X)}{\hat\gamma_1(X)} \right),
\end{align*}
where $X$ is the distinguished representative produced by the plumbing construction for the projective equivalence class with coordinates $(s_2,\ldots,s_n,t)$.   We use the common value for coordinates at lower levels; coordinates at higher levels will only affect the scaling of periods on the right, so the ratios are still well-defined.

We will understand $P$ by controlling its derivative, which is the content of:

\begin{sublemma}
  \label{sublemma:der}
The derivative of $P$ has the following form: 
\begin{align*}
  DP = 
  \left(\begin{matrix}
  d + o(1) &  &  & O(\log t)\\
   & \ddots & & \vdots \\
    & & d + o(1) & O(\log t)\\
   o(t) & \cdots & o(t) & d + o(1)
 \end{matrix}\right),
\end{align*}
where all blank entries are $o(1)$, and $d\ne 0$ only depends on $\bar X$.  Here we use the usual $o(\cdot)$ and $O(\cdot)$ notations to denote quantities that are asymptotically smaller than, respectively, smaller than or equal to, the argument, as $\epsilon \to 0$ (where $\epsilon$ is the parameter used to define $V$ in Section \ref{sec:vk}).  The implied constants depend only on $\bar X$ (not on $S$).
\end{sublemma}

\begin{proof}
  First we write explicit expressions for $\hat\gamma_i(X), \hat\beta(X)$.

By Lemma \ref{lemma:period},
\begin{align*}
  \hat\gamma (X) = t_{-1}^{a_{-1}} \cdots t_{\ell}^{a_\ell} \left( \gamma ^{\pert}(X) + c + t  f(X)+  (t  \log t )g(X)  + t h(X) \right).
\end{align*}

The expression for the period of $\hat \beta$ is simpler, since this class only interacts with subsurfaces at levels $\ell-1$ and below: 
\begin{align*}
  \hat \beta(X) = t_{-1}^{a_{-1}} \cdots t_{\ell}^{a_\ell} \cdot (tb) ,
\end{align*}
where $b$ is a function that is constant on $S$, and such that $b= 1+o(1)$.  

We now aim to estimate the partial derivatives of the components of $P$.  

First we need estimates on how the perturbed periods vary.  At the boundary, the perturbed period is just a fixed translate of the corresponding relative period, which implies that  $\frac{\partial \gamma_i^{\pert}}{\partial s_i}\Big|_{\bar X} = 1,$ for $i\ne 1$,
and $\frac{\partial \gamma_i^{\pert}}{\partial s_j}\Big|_{\bar X} = 0,$ for $i\ne j$.  
Since $\gamma_i^{\pert}$ is a holomorphic function, it follows that at any $X$ for which $[X]\in S$. 
$$\frac{\partial \gamma_i^{\pert}}{\partial s_i} = 1 + o(1), \text{ \ \ for } i \ne 1,$$
$$\frac{\partial \gamma_i^{\pert}}{\partial s_j} = o(1), \text{ \ \ for } i \ne j.$$
Also, since $\gamma_i^{\pert}$ is holomorphic, we have 
$$\frac{\partial \gamma_i^{\pert}}{\partial t} = O(1).$$

Now we proceed to estimate the partial derivatives of components of $P$, evaluated at a point $X$ for which $[X]\in S$:
\begin{itemize}
\item  For $i\ne 1$: 
\begin{align*}
  \frac{\partial}{\partial s_i} \frac{\hat\gamma_i}{\hat\gamma_1} =& \frac{\partial}{\partial s_i} \frac{\gamma ^{\pert}_i + c_i + t  f_i+  (t  \log t )g_i  + t h_i }{\gamma ^{\pert}_1 + c_1 + t  f_1 + (t  \log t )g_1 + t h_1}\\
                                                                  =& \frac{(1+o(1))\left(\gamma ^{\pert}_1 + c_1 + t  f_1 + (t  \log t )g_1 + t h_1\right)}{(\gamma ^{\pert}_1 + c_1 + t  f_1 + (t  \log t )g_1 + t h_1)^2} \\
  &- \frac{o(1) \left(\gamma ^{\pert}_i + c_i + t  f_i+  (t  \log t )g_i  + t h_i\right)}{(\gamma ^{\pert}_1 + c_1 + t  f_1 + (t  \log t )g_1 + t h_1)^2} \\
  =& d+ o(1),
\end{align*}
where $d= \frac{1}{\gamma_1^{\pert}(\bar X)+c_1}$. Note that we can take $\gamma_1^{\pert}(\bar X) + c_1$ to be non-zero, since the disc removed for perturbed period coordinates is small, so $\gamma_1^{\pert}(\bar X)$ is close to $\gamma_1(\bar X) =1$, and $c_1$ can be taken to be small.  

\item For $i\ne j$:
\begin{align*}
  \frac{\partial}{\partial s_i} \frac{\hat\gamma_j}{\hat\gamma_1} =& \frac{\partial}{\partial s_i} \frac{\gamma ^{\pert}_j + c_j + t  f_j+  (t  \log t )g_j  + t h_j }{\gamma ^{\pert}_1 + c_1 + t  f_1 + (t  \log t )g_1 + t h_1}\\
  =& \frac{o(1)\left(\gamma ^{\pert}_1 + c_1 + t  f_1 + (t  \log t )g_1 + t h_1\right)} {\left(\gamma ^{\pert}_1 + c_1 + t  f_1 + (t  \log t )g_1 + t h_1\right)^2}\\
  &- \frac{o(1) \left(\gamma ^{\pert}_j + c_j + t  f_j+  (t  \log t )g_j  + t h_j\right)}{\left(\gamma ^{\pert}_1 + c_1 + t  f_1 + (t  \log t )g_1 + t h_1\right)^2} \\
  =& o(1).
\end{align*}

\item For any $i$:
\begin{align*}
  \frac{\partial}{\partial t} \frac{\hat\gamma_i}{\hat\gamma_1} =& \frac{\partial}{\partial t} \frac{\gamma ^{\pert}_i + c_i + t  f_i+  (t  \log t )g_i  + t h_i }{\gamma ^{\pert}_1 + c_1 + t  f_1 + (t  \log t )g_1 + t h_1}\\
  = &\frac{\left(O(1) + O(\log t) \right)\left(\gamma ^{\pert}_1 + c_1 + t  f_1 + (t  \log t )g_1 + t h_1\right)}{(\gamma ^{\pert}_1 + c_1 + t  f_1 + (t  \log t )g_1 + t h_1)^2} \\
                                                                 & - \frac{\left(\left(O(1) + O(\log t) \right)\right) \left(\gamma ^{\pert}_i + c_i + t  f_i+  (t  \log t )g_i  + t h_i\right)}{(\gamma ^{\pert}_1 + c_1 + t  f_1 + (t  \log t )g_1 + t h_1)^2} \\
  =& O(\log t).
\end{align*}

\item  For $i\ne 1$:
\begin{align*}
  \frac{\partial}{\partial s_i} \frac{\hat\beta}{\hat\gamma_1} &=\frac{\partial}{\partial s_i}  \frac{ bt }{ \gamma ^{\pert}_1 + c_1 + t  f_1 + (t  \log t )g_1 + t h_1 }\\
                                                               &= \frac{ -o(1)bt }{ \left(\gamma ^{\pert}_1 + c_1 + t  f_1 + (t  \log t )g_1 + t h_1\right)^2 } \\
                                                               & = o(t).
\end{align*}

\item Finally:
\begin{align*}
  \frac{\partial}{\partial t} \frac{\hat\beta}{\hat\gamma_1} &=\frac{\partial}{\partial t}  \frac{ bt }{ \gamma ^{\pert}_1 + c_1 + t  f_1 + (t  \log t )g_1 + t h_1 }\\
                                                                   &=\frac{ b(\gamma ^{\pert}_1 + c_1 + t  f_1 + (t  \log t )g_1 + t h_1 ) -  \left(O(1) +O(\log t) \right) bt}{\left(\gamma ^{\pert}_1 + c_1 + t  f_1 + (t  \log t )g_1 + t h_1\right)^2}\\
                                                                   &=d+o(1),                                                                   
\end{align*}
where $d= \frac{1}{\gamma_1^{\pert}(\bar X)+c_1}$, as in the first bullet above (recall also that $b=1+o(1)$).

\end{itemize}

Putting together the above estimates gives the desired form for $DP$.  

\end{proof}

Since each $P_k$ is ratio of periods, to finish the proof in this case, it suffices to find some $k$ such that $P_k([X'])\ne P_k([X''])$.  This is achieved from the following:

\begin{sublemma}
  The map $P: S\to \C^n$ defined above is injective, provided that the set $V$ defined in Section \ref{sec:vk} (which contains  $[X'],[X'']$) is chosen with sufficiently small $\epsilon$. 
\end{sublemma}

\begin{proof}
  Our strategy to prove injectivity of $P$ is to estimate the derivative and then integrate.  For surfaces in $V$, we will work with the distinguished representative $X$ of the projective equivalence class throughout.  Suppose for the sake of contradiction that $[Y'],[Y'']\in S$ satisfy $[Y']\ne [Y'']$ but $P([Y']) = P([Y''])$. We begin by applying Sublemma \ref{sublemma:der}, which tells us the form of the derivative matrix $DP$. 

The idea for the remainder of the proof is as follows.  Notice that for $t$ small, the matrix $DP$ has non-zero determinant (all products, except the product of diagonal entries, are $O\left(t\log t\right)$, which goes to $0$ as $t\to 0$).  If $DP$ were close to a \emph{fixed} invertible matrix, then we can get injectivity by integrating the derivative and using the Fundamental Theorem of Calculus.  However, the $O(\log t)$ entries in $DP$ can blow up as $t\to 0$.  The solution is to divide into two cases. In (\ref{item:diffmag}) we assume that the values $t',t''$ for $Y',Y''$ have quite different magnitudes and then directly show that $P_n$ takes different values.  In (\ref{item:samemag}) we assume that $t',t''$ have somewhat similar magnitudes, in which case the relevant values of $DP$ are in fact close to a fixed invertible matrix and so we can use the integration argument.

So we consider two cases:

\begin{enumerate}[(I)]
\item \label{item:diffmag}
  Suppose $|t''|/|t'| \notin [1/2,2]$.
We can assume without loss of generality that $|t'|> 2|t''|$.  In this case, we don't need the derivative estimate.  Note that 
   \begin{align*}
     |P_n(Y')| = \left|\frac{bt'}{\gamma_1^{\pert}(Y') +r_1t'\log t' + t'c_1} \right| > 0.9 |t'| \frac{|b|}{|\gamma_1^{\pert}(\bar X)|} > 1.8|t''| \frac{|b|}{|\gamma_1^{\pert}(\bar X)|} ,
   \end{align*}
when $V$ (hence $t'$) is small.  Similar estimates give $|P_n(Y'')| < 1.1 |t''| \frac{|b|}{|\gamma_1^{\pert}(\bar X)|} $ for $V$ small, so $P_n(Y')\ne P_n(Y'')$, and we are done.  

\item \label{item:samemag}
  Suppose $|t''|/|t'| \in [1/2,2]$.  We will consider the directional derivative in the direction $v=(v_1,\ldots,v_{n-1}, \tau)^T$.  Multiplying by the matrix $DP$ and using Sublemma \ref{sublemma:der} gives the following expressions for the directional derivatives of the component functions of $P$:
 \begin{align*}
   \nabla_v(P_i) &= dv_i + o(1) \max_i |v_i| + \tau \cdot O(\log t), \text{ \ \  for } i=1,\ldots,n-1, \\
   \nabla_v(P_n) &= o(t) \max_i|v_i| + \tau(d+o(1)),
 \end{align*}
where $d\ne 0$ is some complex number that only depends on $\bar X$.  
 Now $\P V$ is a convex region (with respect to the analytic coordinates), and hence $S$ is also convex.  So we can connect $Y',Y''$ by a straight line segment $L$ lying in $S$ (and $L$ has non-zero length, since $Y',Y''$ have different analytic coordinates).  Let $v$ be the direction of this segment. 

 Note that for every surface $X$ on the segment $L$, we have
 \begin{align*}
   |t| \in \left[ \frac{|t'|}{2}, 2|t'| \right],
 \end{align*}
 since all such $t$ lies in a sector of angle at most $\pi/4$, by conditions (\ref{item:horiz}) and (\ref{item:vert}) in Section \ref{sec:vk}.  (The distance from $0$ to $t$ is a convex function along $L$, so $|t|\le \max(|t'|,|t''|) \le 2|t'|$.  If the triangle formed by $0,t',t''$ has an obtuse angle, which cannot be the angle at $0$, then $|t|$ is monotone along the segment connecting $t',t''$, and hence $|t|\ge \min(|t'|,|t''|) \ge |t'|/2$.  If the triangle is acute, then  $|t|= |t'| \cos\theta$ for some $\theta \in [0,\pi/4]$, so $|t|\ge |t'|\sqrt{2}/2\ge |t'|/2$.)
   
   We now break into two further subcases depending on the size of $\max_i|v_i|$, which we take to be $|v_j|$:
   \begin{enumerate}[(a)]
   \item \label{item:a}
     $|v_j|  \ll \frac{|\tau|}{|t'|}.$  We then get
     \begin{align*}
       \nabla_v(P_n)|_X &= o(t) |v_j| + \tau(d+o(1)) = o(2t') |v_j| + \tau(d+o(1))\\
       &= \tau (d + o(1)). 
     \end{align*}
     More precisely, the above means that for each $\kappa>0$, there is some constant $c>0$ such that if $|v_j|  < c \frac{\tau}{|t'|},$ then
     \begin{align}
       \big|\nabla_v(P_n)|_X - d\tau\big| < \kappa |d\tau|. \label{eq:a}
     \end{align}
     In particular, we can take the constant $c$ for $\kappa=1/2$.  This $c$ depends on the implied constant in $o(\cdot)$ above (which depends on $\bar X$, but not on $S$).  
     
   \item \label{item:b}
     $|v_j| \gg  |\tau| \cdot |\log t'|$.  Then
     \begin{align*}
       \nabla_v(P_j)|_X &= dv_j + o(1) |v_j| + \tau \cdot O(\log 2 t') = v_j(d+ o(1)).  
     \end{align*}
     More precisely, the above means that for each $\kappa>0$, there is some constant $C>0$ such that if $|v_j| > C \tau|\log t'|,$ then
     \begin{align}
       \big|\nabla_v(P_j)|_X - dv_j\big| < \kappa |dv_j|.\label{eq:b}
     \end{align}
     In particular, we can take the constant $C$ for $\kappa=1/2$.  This $C$ depends on the implied constant in $O(\cdot)$ above  (which depends on $\bar X$, but not on $S$).  
   \end{enumerate}

   We now choose $\epsilon$ in the definition of $V$ in Section \ref{sec:vk} small depending on the constants $c,C$ above.  Specifically, we take $\epsilon$ such that $|\epsilon \cdot \log \epsilon| < c/C$ (note that the LHS goes to zero as $\epsilon \to 0$) for the pair $(c,C)$ that we get in the above for $\kappa=1/2$.  Note that $c,C$ depend only on $\bar X$, not $S$.  Since $Y'\in V$, we have that $|t'|< \epsilon$, hence $|t'|\cdot |\log t'| < c/C$, which means $C\tau |\log t'| < c \frac{\tau}{|t'|}$, and so the subcases (a) and (b) cover all possibilites for $|v_j|$. 
   
Now define $z,k$ as follows.  In case (a), we take $z=d\tau$ and $k=n$.  In case (b), take $z=dv_j$ and $k=j$ (if $|v_j|$ falls into both cases, choose one arbitrarily).   Note that $z\ne 0$, since some $v_i$ or $\tau$ must be non-zero (and $d\ne 0$), and then the inequality for (a) or (b) implies that $z$ is also non-zero.   Consider the open half-plane $H$ in $\C$ whose boundary is the line through the origin perpendicular to $z$.  From the inequalities \eqref{eq:a} or \eqref{eq:b} above, we get that $\nabla_v(P_k)|_X$ lies in $H$ for all $X\in L$.  Then by the Fundamental Theorem of Calculus, the difference $P_k(Y'')-P_k(Y')$ can be expressed as an integral of $\nabla_v(P_k)|_X$ over the segment $S$.  Since the integrand lies in $H$, and $H$ is closed under addition, we get that $P_k(Y'')-P_k(Y')\in H$.  In particular it is non-zero, hence $P(Y')\ne P(Y'')$, and we are done. 

\end{enumerate}

\end{proof}

\medskip \noindent \underline{Case 3:} \emph{$\ell$ is not an integer (i.e. the lowest effective level where the two surfaces differ comes from a degenerating cylinder).}

Here we must understand the period of a curve crossing the degenerating cylinder that witnesses the effective level.

Let $j=\ell-1/2$.  All the moduli parameters for the level subsurfaces $\bar X^{(j)}$ and below are the same for $X',X''$.  For the scaling parameters, we have $t_i'=t_i''$ for $i\le j-1$.  By assumption, there is some degenerating cylinder at level $j$ for which the parameter values $t',t''$ for $X',X''$, are different.  On surfaces obtained from $\bar X^{(j)}$ by smoothing out the horizontal node, we can define a relative homology class $\alpha$ that crosses this cylinder once.  We extend $\alpha$ downwards to a relative homology class $\hat\alpha$ on smooth surfaces in $V$, and we can do this in such a way that $\hat\alpha$ does not cross any other degenerating cylinders (as in Case 2).  We take $\gamma_1$ to be the class on $\bar X^{(j)}$ that has period $1$ on all surfaces near $\bar X$ in the boundary stratum.   This can be extended downwards to a relative homology class $\hat\gamma_1$ on smooth surfaces in $V$.

In the plumbing construction to smooth the horizontal node, plumbing is performed on the discs of radius $1$ in standard coordinates centered at the nodes.  The class $\hat\alpha$ can be decomposed into a piece that lies outside this disc; the period of this piece does not depend on $t$.  There is also a portion that crosses the plumbing region.  The period here is equal to $\int_1^{\sqrt{t}} \frac{r}{u} du =(r/2) \log t $, where $r$ is the residue at the simple pole.  Since we are working in a sector, $\log$ can be consistently defined.  

We then get 
$$\frac{\hat\alpha(X')}{\hat\gamma_1(X')} = \frac{(t_{-1}')^{a_{-1}} \cdots (t_{j}')^{a_j}(c + (r/2) \log t')}{(t_{-1}')^{a_{-1}} \cdots (t_{j}')^{a_j}d} = \frac{c+(r/2)\log t'}{d},$$
$$\frac{\hat\alpha(X'')}{\hat\gamma_1(X'')} = \frac{(t_{-1}'')^{a_{-1}} \cdots (t_{j}'')^{a_j} (c + (r/2) \log t'')}{(t_{-1}'')^{a_{-1}} \cdots (t_{j}'')^{a_j} d} = \frac{c+(r/2)\log t''}{d}.$$
It is crucial in the above that $c,r,d$ are the same for $X',X''$; this is because these depend only on moduli parameters at level $j$ and below, and level scaling parameters for level $j-1$ and below, and all these parameters are the same for $X',X''$.

So, since we are assuming that $t' \ne t''$, we get $\frac{\hat\alpha(X')}{\hat\gamma_1(X')} \ne
\frac{\hat\alpha(X'')}{\hat\gamma_1(X'')},$ and we are done.  

\end{proof}

\section{Intersecting $V$ with $\M$}
\label{sec:intersect}
We now introduce the affine invariant manifold $\M$ into the picture.  Our goal is to show that $\M$ intersects each period coordinate chart $V$ constructed in Section \ref{sec:vk} nicely.  We will use the following deep result of Filip.  

\begin{thm}[\cite{filip2016}]
    \label{thm:filip}
  Any affine invariant manifold $\M$ is a quasi-projective subvariety of $\H$.  
\end{thm}

\begin{cor}
  \label{cor:semi}
 $\P\M$ is a semianalytic subset of $\barPH$, thought of as a real-analytic orbifold.  
\end{cor}

\begin{proof}
  This would be easier if $\barPH$ was a projective algebraic variety; however this is currently unknown.  Instead, we will use
  \begin{enumerate}[(i)]
  \item the fact that we have available the \emph{Incidence Variety Compactification} $\barPH^{\text{IVC}}$ of $\P\H$ \cite{bcggm2018}, which is a projective variety, and    
  \item the existence of a surjective map $\pi:\barPH \to \barPH^{\text{IVC}}$ that is a complex-analytic (hence real-analytic) map of complex-analytic spaces.  See \cite[Theorem 1.2 (5)]{bcggm2019}.
  \end{enumerate}

  By Theorem \ref{thm:filip}, $\pi(\P\M)$ is a quasi-projective subvariety of $\barPH^{\text{IVC}}$.  Furthermore, $\pi(\P\M) = W-Z$, where $W$ is the Zariski closure of $\pi(\P\M)$ in $\barPH^{\text{IVC}}$, and $Z = \barPH^{\text{IVC}} -\pi(\P\H)$, which is also an algebraic variety.  Then we have
  $$\P\M = \pi^{-1}(\pi(\P\M))=\pi^{-1}(W-Z) = \pi^{-1}(W)-\pi^{-1}(Z).$$
  Since $W,Z$ are both algebraic subvarieties of the IVC, their preimages under the real-analytic map $\pi$ are real-analytic varieties, and hence their difference is semianalytic.  
 \end{proof}

 The following basic facts give finiteness of components of semianalytic sets.  
 
\begin{fact}
  \label{fact:local-components}
  Let $S$ be a semianalytic subset of a real-analytic orbifold $M$.  Then for any $p\in M$, and $W'$ a neighborhood of $p$ in $M$, we can pass to a smaller neighborhood $W\subset W'$, $p\in W$, such that $W \cap S$ has finitely many connected components.
\end{fact}

\begin{proof}
  This is proved in \cite[proof of Corollary 2.7]{bm1988} for the case when $M$ is a real-analytic manifold.  The generalization to orbifolds is straightforward.  See also \cite[Theorem 6.4,  part (7)]{kankaanrinta2011} for a related statement.  
\end{proof}

\begin{fact}
    \label{fact:components}
  Let $M$ be a \emph{compact} real-analytic orbifold, and let $S\subset M$ be a semianalytic subset. Then $S$ has finitely many components.
\end{fact}
  
\begin{proof}
  By Fact \ref{fact:local-components}, for each $p\in M$, we can find an open set $W_p\subset M$, $p\in W_p$, such that $W_p\cap S$ is a union of finitely many connected sets.  By compactness of $M$, we can cover $M$ by finitely many of these, say $W_{p_1},\ldots,W_{p_k}$.  So $S$ is a union of finitely many connected sets, and it follows that it has finitely many connected components.  
\end{proof}

We now apply the above facts to $V\cap \M$.  

\begin{lemma}[Finite intersections]
    \label{lemma:finite-intersections}
  The intersection of the smooth locus $\M^* \subset \M$ with each set $V$ from Section \ref{sec:vk} has finitely many connected components.
\end{lemma}

\begin{proof}
  First we show that $\P V \cap \P\M^*$ has finitely many components.  By Lemma \ref{lemma:semianalytic}, $\P V$ is a semianalytic subset of $\barPH$.  By Corollary \ref{cor:semi}, $\P\M$ is a semianalytic subset of $\barPH$, and hence $\P\M^*$ also is (since the singular locus of variety is a subvariety).   Thus the intersection $\P V \cap \P\M^*$ is semianalytic.  Finiteness of components then follows by applying Fact \ref{fact:components}.  

  The connected components of $V \cap \M^*$ are in bijection with those of $\P V \cap \P\M^*$, hence also finite.  
\end{proof}

We now show that, in our volume estimation problem, $\mu_\M$ can be replaced by a finite sum of linear measures on subspaces $W$, and $\M$-independence of saddle connections can be replaced  by $W$-independence, defined below.  

Fix a period coordinate chart $Q\subset \H$, and $W$ a subset of $Q$ given by pulling back a single linear subspace from $\C^n$.  Let $W^1\subset W$ be the locus of unit area surfaces.  We define
\begin{align*}
  L_{\epsilon_1,\ldots,\epsilon_k}^W := \{&X\in W^1:  X \text{ has } W\text{-independent saddle connections } s_1,\ldots,s_k,\\
  &\text{ with } |s_i| \le \epsilon_i \}.
\end{align*}

Here a collection of saddle connections on $X\in W$ is said to be \emph{$W$-independent} if their relative homology classes define linearly independent functionals on $W$.  

We define a measure $\mu_W$ on $Q$ to be the natural Lebesgue measure on the linear subspace $W$ (we can pick an arbitrary normalization).  Let $\mu_W^1$ be the corresponding measure supported on the unit area locus $W^1$, i.e. for any measurable subset $S\subset Q^1$, 
$$\mu_W^1(S) : = \mu_W\left( \{ sX : X \in S, \text{ } 0\le s \le 1\} \right).$$

\begin{lemma}[Finitely many subspaces]
  \label{lemma:finite-subspaces}
  Fix a $V$ from Section \ref{sec:vk} satisfying Lemma \ref{lemma:period-chart}, and an affine invariant submanifold $\M$.  There exist finitely many linear subspaces $W_1,\ldots,W_j \subset V$, and a constant $C$, such that for any $\epsilon_1,\ldots,\epsilon_k >0$, 
$$\mu_{\M}^1\left(L_{\epsilon_1,\ldots,\epsilon_k}^\M \cap V\right) \le C \cdot \sum_\ell \mu_{W_\ell}^1\left(L_{\epsilon_1,\ldots,\epsilon_k}^{W_\ell}\right).$$
\end{lemma}

\begin{proof} 
  By Lemma \ref{lemma:finite-intersections} (Finite intersections), the smooth locus $\M^*$ intersects $V$ in finitely many components $N_1,\ldots,N_j$.  Now, by the definition of an affine invariant manifold, locally near a point of $N_\ell$, we have that $\M^*$ agrees with a linear subspace $W_\ell\subset V$ since $V$ is a period coordinate chart, by Lemma~\ref{lemma:period-chart}.  Furthermore,  the measures $\mu_{\M}$ and $\mu_{W_\ell}$ agree locally, up to a constant scaling factor $C_{\ell}$.  The subset of $N_\ell$ where $N_\ell$ agrees locally with $W_\ell$ is open and closed in $N_\ell$.  Since $N_\ell$ is connected, this means that this set of local agreement is all of $N_\ell$.  Hence $N_\ell \subset W_\ell$.  And near points in $N_\ell$, the measures $\mu_{\M}$ and $\mu_{W_\ell}$ agree up to the factor $C_{\ell}$, and furthermore the notion of saddle connections being $\M$-independent coincides with the notion of $W_\ell$-independence at such points.  Hence
$$\mu_\M^1\left(L_{\epsilon_1,\ldots,\epsilon_k}^\M \cap N_\ell\right) \le C_{\ell}\cdot  \mu_{W_\ell}^1 \left(L_{\epsilon_1,\ldots,\epsilon_k}^{W_\ell}\right).$$
Using these observations, and the fact that $\mu_{\M}(\M-\M^*)=0$, we get that
$$\mu_{\M}^1\left(L_{\epsilon_1,\ldots,\epsilon_k}^\M \cap V\right) \le \sum_\ell \mu_{\M}^1\left(L_{\epsilon_1,\ldots,\epsilon_k}^\M \cap N_\ell\right) \le \sum_\ell C_\ell \cdot \mu_{W_\ell}^1 \left(L_{\epsilon_1,\ldots,\epsilon_k}^{W_\ell}\right),$$
and we get the desired result by taking $C=C_1 + \cdots + C_\ell$.  

\end{proof}

\section{Local volume bound}
\label{sec:vol-comp}

Below is the key local (near a boundary point) result needed to prove Theorem \ref{thm:regularity}

\begin{prop}[Local volume bound]
    \label{prop:local-regularity}
  Let $\bar X \in \barPH$.  Then there exists a small neighborhood $U\subset \barPH$ containing $\bar X$ and a constant $C$ such that 
  $$\mu_\M^1 (L_{\epsilon_1,\ldots,\epsilon_k}^\M \cap p^{-1}(U)) \le C \epsilon_1^2 \cdots \epsilon_k^2,$$
  for any $\epsilon_1,\ldots, \epsilon_k>0$.
\end{prop}

First we will develop the necessary tools to understand the relationship between short saddle connections and the structure of the boundary point $\bar X$.  This culminates in Lemma \ref{lemma:vol-subspace}, which is the above estimate but for a single linear subspace $W$ in the period coordinate chart $V$.    Then in Section \ref{sec:proof-local} we will combine this estimate with Lemma \ref{lemma:finite-subspaces} (Finitely many subspaces) to prove the above proposition.  

We will work with surfaces $X$ in a period coordinate chart $V$ from Section \ref{sec:vk} that covers part of a neighborhood of a boundary point $\bar X$.  We need to understand where the short saddle connections on $X$ are in terms of data from $\bar X$.   One challenge is understanding the interaction of degenerating cylinders with small subsurfaces.  

\subsection{Sizes of subsurfaces and orderings}
\label{sec:sizes-subsurfaces}
Recall that the boundary point $\bar X$ has level subsurfaces $\bar X^{(i)}$ consisting of all those components at level $i$.  Each of these corresponds to a subsurface of each $X\in V$, defined up to isotopy.  For each horizontal node (corresponding to a pair of simple poles) there is a \emph{degenerating cylinder} on $X$, which is also a subsurface defined up to isotopy.  The boundaries of all these subsurfaces and degenerating cylinders of $X$ consist of curves that are pinched in $\bar X$.  

Let $\mathcal{S}$ be the union of the set of the level subsurfaces and the set of degenerating cylinders.  (Note that we do not excise the degenerating cylinders from the level subsurfaces; doing so would not actually change any of the below in a substantial way).  Each element of $\mathcal{S}$ is a topological subsurface, defined up to isotopy.  On a particular $X\in V$, we can talk about a definite size of each element of $\mathcal{S}$:

\begin{defn}[Size]
Given a surface $X\in V$, we define $\operatorname{size}_X(Y)$ of an element $Y\in \mathcal{S}$ to be
\begin{enumerate}[(i)]
\item $\left|t_{-1}^{a_{-1}}\cdots t_i^{a_i}  \right|$ if $Y$ is the level $i$ subsurface $\bar X^{(i)}$; this is the magnitude of the scaling parameter in the plumbing construction for the projective class of $X$ (see Section~\ref{sec:plumb}) 
\item $\left|\log |t| \cdot t_{-1}^{a_{-1}}\cdots t_i^{a_i} \right|$ if $Y$ is a degenerating cylinder with horizontal node parameter $t$ whose circumference lies at level $i$ (this is approximately the height of the cylinder, since the $1/u$ pole of the differential contributes $(\log u)|_1^{\sqrt{t}}$ to the integral of the curve crossing the annulus where the plumbing happens).  
\end{enumerate}
\end{defn}

\paragraph*{Ordering of subsurfaces.}
We wish to use the structure of the boundary point $\bar X$ to understand something about the relative size of saddle connections on surfaces $X\in V$.  There is a natural ordering on the set of level subsurfaces, just given by the level.  If there are no degenerating cylinders, the ordering induced on saddle connections from their level roughly agrees with the orderings of their lengths on surfaces in $V$.  However, for saddle connections that cross degenerating cylinders at lower levels the situation is  more complicated.  

We wish to extend the ordering on level subsurfaces to an order
on all of $\mathcal{S}$, including the degenerating cylinders.  Since the expression in (ii) above contains both large terms ($|\log |t||$), and small terms (the others), the limit point $\bar X$ does not by itself tell us about the relative magnitudes of degenerating cylinder cross curves compared to other subsurfaces of $X$.  Thus, we have several possible orderings compatible with a given $\bar X$; however, since there are only finitely many, our strategy will be to do a separate volume computation for each ordering, and then take a sum.

Let $\mathcal{O}$ denote the set of total orderings $\succ$ on $\mathcal{S}$ that restrict to the natural ordering on the level subsurfaces and have the additional property that if $C$ is a degenerating cylinder and $\bar X^{(i)}$ is the level subsurface at which the circumference of $C$ lies, then $C\succeq \bar X^{(i)}$.

We say that an ordering $\succ$ in $\mathcal{O}$ is \emph{consistent} with $X$ if for any $Y_1,Y_2\in \mathcal{S}$ with $Y_1\succ Y_2$ we have that $\operatorname{size}_X(Y_1) \ge \operatorname{size}_X(Y_2)$.  For each $X$, there is at least one $\succ \in \mathcal{O}$ that is consistent with $X$.

Given a relative integral homology class $\gamma$ on surfaces in $V$ and $\succ\in \mathcal{O}$, we define the $\succ$-\emph{level} of $\gamma$ to be the minimal $Y\in \mathcal{S}$ with respect to $\succ$ such that $\gamma$ has a representative that does not intersect any $Y'\in \mathcal{S}$ with $Y'\succ Y$.
Here we consider those representatives of $\gamma$ that are unions of oriented arcs/curves (such a representative always exists).

We will call a degenerating cylinder $\succ$-\emph{wide} if it is $\succ$-greater than any level subsurface.

\subsection{Lemmas controlling periods}

In this section we prove bounds on the periods of various homology classes in terms of the sizes of the subsurfaces that they intersect.   Here we are only interested in coarse bounds (constant factors do not matter).   Parts of these lemmas could be proven using the more delicate estimates in Section \ref{sec:tools} (though the estimates there do not handle homology classes crossing degenerating cylinders).  Since this precision is not necessary for our purposes here, we do not make use of those more precise estimates in this section.  

Our first lemma states that the period of a relative homology class is coarsely bounded above by the size of its $\succ$-level.  A similar estimate has been used recently by Chen-Wright \cite[Theorem 8.1]{cw2019}.

\begin{lemma}[Period bounds]
  \label{lemma:period-bounds}
 For $\gamma$ a relative (integral) homology class defined on the surfaces in $V$, there exists a constant $C>0$ (depending only on $\gamma$ and $V$) with the following property.    For any $X\in V$ of area $1$ and any $\succ \in \mathcal{O}$ consistent with $X$, we let $Y$ be the $\succ$-level of $\gamma$.  Then
  $$|\gamma(X)| \le C\cdot \operatorname{size}_X(Y).$$
\end{lemma}

\begin{proof}
  
  Since $X$ is in $V$, there is some $\C$-rescaling $X'=\alpha X = (M,\omega)$ that is the distinguished surface coming from the plumbing construction.  We will work with this $X'$ and then compare to $X$ at the end.  By definition, $\operatorname{size}_X(Y)=\operatorname{size}_{X'}(Y)$ for any subsurface $Y\in \mathcal{S}$.

  The underlying Riemann surface $M$ can be decomposed into pieces corresponding to elements of $\mathcal{S}$. For a level subsurface $\bar X^{(i)}$ in $\mathcal{S}$, the corresponding piece of $M$ is very close to the complement of a fixed small neighborhood of the nodes in $\bar X^{(i)}$.  For a degenerating cylinder $C$ in $\mathcal{S}$ whose circumference lies at level $i$, the corresponding piece of $M$ is an annulus.  These pieces do not quite cover all of $M$; there are also \emph{plumbing annuli} for the nodes connecting level subsurfaces.  See Figure \ref{fig:surface_pieces}.   

   \begin{figure}[]
\begin{center}
  \includegraphics[scale=0.55]{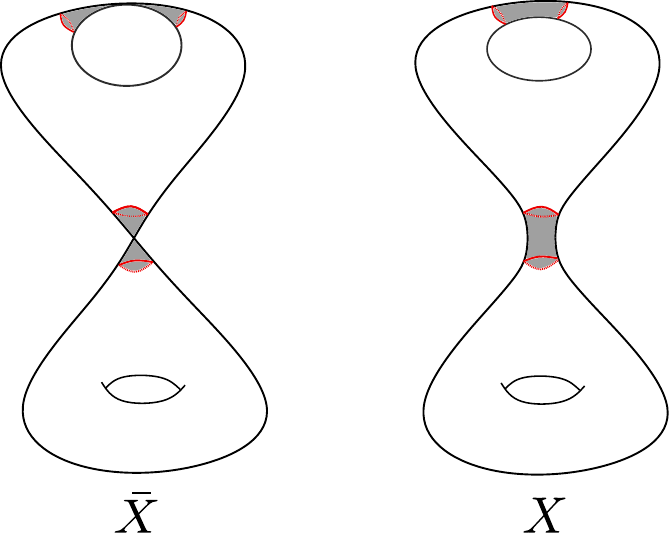}
  \caption{Neighborhoods of nodes of $\bar X$ are replaced by plumbing annuli and degenerating cylinders to get $X$.  }
  \label{fig:surface_pieces}
\end{center}
\end{figure}

  Now for $\gamma$, choose representative curves/arcs (which we will also call $\gamma$) on the surfaces in $V$ that exhibit the $\succ$-level $Y$ of $\gamma$, as defined above in Section \ref{sec:sizes-subsurfaces}.  The intersection of $\gamma$ with a piece of $M$ that corresponds to a level subsurface can be taken to be a union of fixed arc (with endpoints at holomorphically varying base points; see Section \ref{par:standard} and beginning of proof of Lemma \ref{lemma:period}).  Because of the way $V$ was defined in terms of restrictions on the plumbing coordinates, $\gamma$ winds around each plumbing annulus a bounded number of times (with the bound only depending on the boundary point $\bar X$, not $X$).  See Figure~\ref{fig:cylinder_preplumb}.  

  Now we consider the restriction of the differential $\omega$ to the pieces of $M$.
  
  \medskip \noindent \underline{Level subsurface:}

  On a piece $\hat Z$ of $X'$ corresponding to the level subsurface $Z=\bar X^{(i)}$, the restriction of $\omega$ gives a translation surface that is close to a piece of $\bar X^{(i)}$ scaled by $t_{-1}^{a_{-1}}\cdots t_i^{a_i}$.  The (small) discrepancy between the period of $\gamma$ on $\hat Z$ versus on $Z$ (suitably rescaled) is caused by:  
  \begin{enumerate}[(i)]
  \item Moduli parameters.  This error has order at most  $\left| t_{-1}^{a_{-1}}\cdots t_i^{a_i} \right|$. 
  \item Modification differential(s) coming from residues of poles on lower components, and the remerging of zeros.  The order of the contribution from a modification differential added to account for residue on $\bar X^{(j)}$ is at most
    $$\left| t_{-1}^{a_{-1}}\cdots t_i^{a_i} \cdot t_{i-1}^{a_{i-1}}\cdots t_j^{a_j} \right| = o\left(\left|t_{-1}^{a_{-1}}\cdots t_i^{a_i}\right|\right).$$
  \end{enumerate}
It follows that the restriction of $\gamma$ to such a piece has period bounded above by a constant multiple of $\left| t_{-1}^{a_{-1}}\cdots t_i^{a_i} \right|$, which equals $\operatorname{size}_X(Z)$.

\medskip \noindent \underline{Degenerating cylinder:}
~ \\ 
\indent On a piece $\hat C$ of $X'$ corresponding to a degenerating cylinder $C\in \mathcal{S}$ with parameter $t$ whose circumference lies at level $i$, the restriction of $\omega$ gives a cylinder of height on the order of $\left| \log |t| \cdot t_{-1}^{a_{-1}}\cdots t_i^{a_i} \right|$ and circumference on the order of $\left| t_{-1}^{a_{-1}}\cdots t_i^{a_i} \right|$.  We saw above that $\gamma$ winds arounds the cylinder a bounded number of times, so the restriction of $\gamma$ to the cylinder has period bounded above by a constant multiple of $\left| \log |t| \cdot t_{-1}^{a_{-1}}\cdots t_i^{a_i}\right|$, which is $\operatorname{size}_X(C)$.

\medskip \noindent \underline{Plumbing annulus:}

Finally, consider a plumbing annulus $A$ such that the higher adjacent piece $\hat Z$ corresponds to the level subsurface $Z=\bar X^{(i)}$, while the lower adjacent piece is the level subsurface $\bar X^{(j)}$.  There are two contributions to the period of  $\gamma$ restricted to $A$:
  \begin{enumerate}[(i)]
  \item Winding around the annulus.  As discussed above, the winding happens only a bounded number of times.   The contribution to the period is of order at most $|t_{-1}^{a_{-1}}\cdots t_i^{a_i}|$. 
  \item Crossing the annulus.  In the general case where there are non-zero residues (and hence modification differentials), the total contribution is at most
    $$O\left(\left|t_{-1}^{a_{-1}}\cdots t_i^{a_i}\right|\right).$$
    Included in the above is the contribution from the modification differential itself, which is order at most
    $$\left| t_{-1}^{a_{-1}}\cdots t_i^{a_i} \cdot \left(t_{i-1}^{a_{i-1}}\cdots t_j^{a_j} \log \left| t_{i-1}^{a_{i-1}/b}\cdots t_j^{a_j/b} \right| \right)\right| = o\left(\left|t_{-1}^{a_{-1}}\cdots t_i^{a_i}\right|\right).$$
      The log term comes from integrating the residue part $t_{i-1}^{a_{i-1}}\cdots t_j^{a_j} \frac{r}{u}$ of the differential over the plumbing region, which is modeled by an annulus between the discs of radius $1$ and $|T|$, where $T=t_{i-1}^{a_{i-1}/b}\cdots t_{j}^{a_{j}/b}$.  For a similar estimate, see proof of Lemma \ref{lemma:period} (in particular, the estimate for types of part (ii) in that proof).
  \end{enumerate} 
 Hence the total contribution to $\gamma$ from $A$ is of order at most $|t_{-1}^{a_{-1}}\cdots t_i^{a_i}|=\operatorname{size}_X(Z)$.

\bigskip

  We have now seen that for all the pieces of $X'$, if $\gamma$ intersects that piece, the period over the intersection is bounded above by a constant multiple of $\operatorname{size}_X(Z)$ for some $Z \in \mathcal{S}$ that $\gamma$ also intersects.  Since $\gamma$ is a curve/arc that exhibits that its relative homology class is at $\succ$-level $Y$, any such $Z$ satisfies $Z\prec Y$.
  And since $\succ$ is consistent with $X$ (and $X'$), it follows that the period of each such intersection is bounded above by a constant multiple of $\operatorname{size}_X(Y)$.  Putting all the parts of $\gamma$ together gives 
  \begin{align}
    \label{eq:xprime}
    |\gamma(X')| \le C\cdot \operatorname{size}_X(Y),
  \end{align}
For some constant $C$ depending only on $\bar X$.  
   Now we claim that the $\alpha$ from the beginning of the proof is bounded below in magnitude, with the bound depending only on $\bar X$.  This follows from the fact that the area of $X'=\alpha X$ is at least (approximately) the area of the differential $\eta_0$ on the top level subsurface $\bar X^{(0)}$, which is bounded below, while $X$ is assumed to have area $1$.  We then get from \eqref{eq:xprime} that
   $$|\gamma(X)| \le \frac{1}{|\alpha|}|\gamma(X')| \le C\cdot \operatorname{size}_X(Y),$$
for some new choice of $C$, completing the proof.  

\end{proof}

  \begin{figure}[]
\begin{center}
  \includegraphics[scale=0.42]{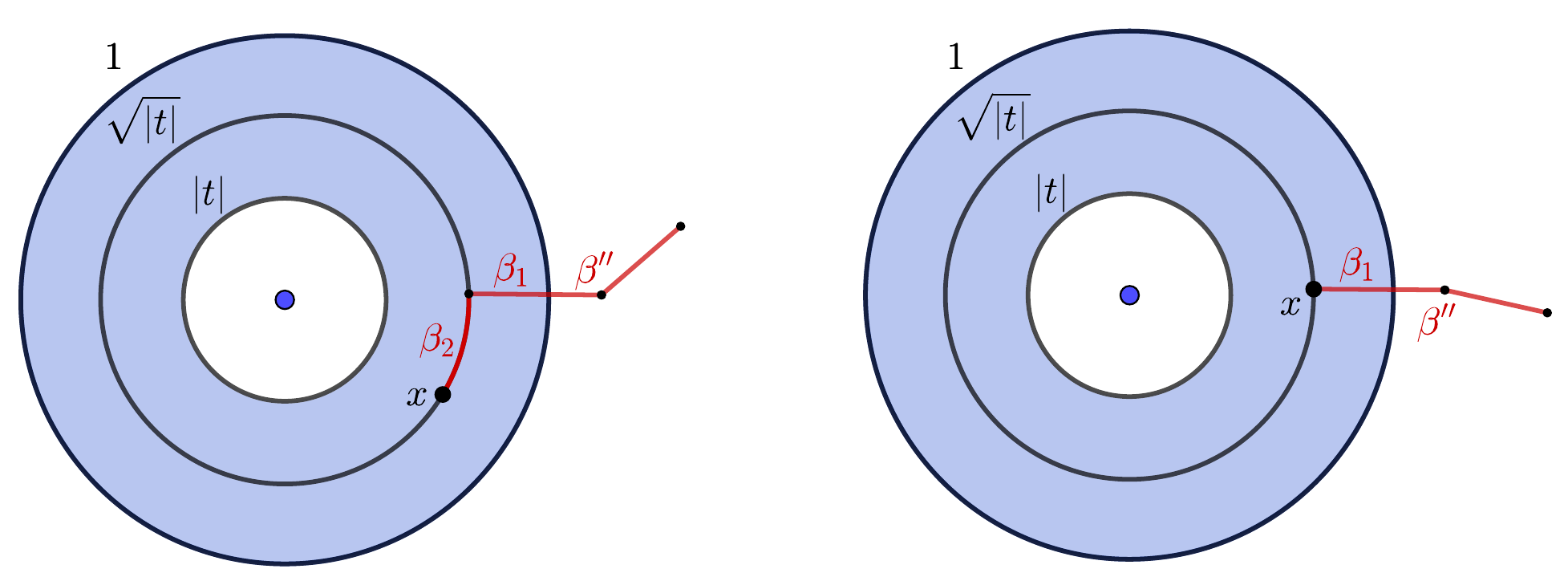}
  \caption{The two annuli are glued together, with the circles of radius $\sqrt{|t|}$ being glued to one another.  The two points labeled $x$ get identified.  The red segments form the part of a curve that crosses the degenerating cylinder. }
  \label{fig:cylinder_preplumb}
\end{center}
\end{figure}

Our next lemma allows us to control the period of a curve that crosses a single degenerating cylinder.  This period is potentially large.  But we show that if we fix the period of the circumference curve of the cylinder, then the period of the crossing curve lies in a rectangle in $\C$ of bounded area.  

\begin{lemma}[Cylinder bounds]
  \label{lemma:cyl-bounds}
  Let $\beta$ be a relative (integral) homology class defined on the surfaces in $V$ with the following properties:
  \begin{enumerate}[(i)]
  \item $\beta$ has a representative that crosses exactly one closed curve $\gamma$ that is the core curve of a degenerating cylinder.  We let $Y$ be the level subsurface containing $\gamma$.  
  \item $\beta$ has a representative such that the level subsurfaces which the representative intersects all lie at or below the level of $Y$ (here we can use $\succ$-level for any $\succ\in \mathcal{O}$, since we are only concerned with the level subsurfaces crossed by $\beta$). 
  \end{enumerate}
Then there exists a family of rectangles $\mathcal{R}(w)\subset \C$, $w\in \C$, of area bounded above by some $R$ (depending on $\bar X$ and $\beta$, but not on $X$), with the following property.  For any $X\in V$ of area $1$, we have
  $$\beta(X)\in \mathcal{R}(\gamma(X)).$$
Furthermore, the rectangles have the property that $s\mathcal{R}(w) \subset \mathcal{R}(sw)$ for any $w$ and $0< s \le 1$.  

\end{lemma}

\begin{proof}
  The rectangle $\mathcal{R}(w)$ is defined as follows:

    \begin{align*}
      \mathcal{R}(w)=\mathcal{R}(re^{i\theta}) := \left\{e^{i (3\pi/2+\theta)}  (x+iy) : |x|\le \frac{1}{|w|} + C'|w|, \text{ } |y| \le (1+C')|w| \right\},
    \end{align*}
    where $C'$ is a constant depending on $\beta$ and $\bar X$, but not $X\in V$, that will be defined later in the proof.  
    Notice that $\mathcal{R}(w)$ gets longer and thinner as $w$ gets smaller, and has approximately fixed area independent of $w$ (when $w$ is small).   
\begin{figure}[]
\begin{center}
  \includegraphics[scale=0.56]{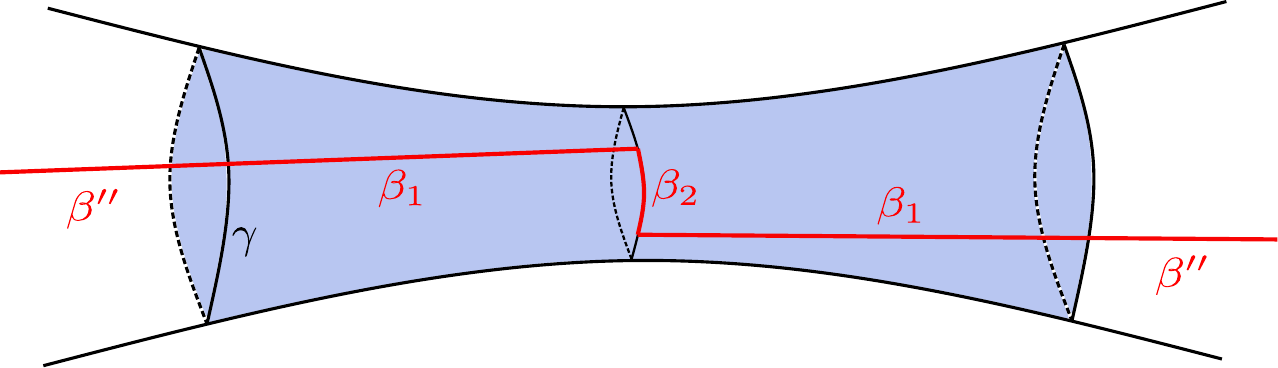}
  \caption{The result of plumbing together the annuli in Figure \ref{fig:cylinder_preplumb}.}
  \label{fig:cylinder_postplumb}
\end{center}
\end{figure}

  As in the proof of Lemma~\ref{lemma:period-bounds} (Period bounds), we divide up $X$ into pieces corresponding to elements of $\mathcal{S}$, and then take the intersection of $\beta$ with these pieces.  Let $\beta'$ be the piece that crosses the cylinder with circumference $\gamma$, and let $\beta''$ be the remaining part of $\beta$.   See Figure \ref{fig:cylinder_preplumb} and Figure \ref{fig:cylinder_postplumb}.  

  To control $\beta'$, we can assume that the situation is as in Figure \ref{fig:cylinder_preplumb}, where the differential on the left is given by $\frac{\gamma(X)}{2\pi i} \frac{du}{u}$, and on the right by $-\frac{\gamma(X)}{2\pi i} \frac{dv}{v}$ (since $\gamma$ is the cylinder circumference curve, the residues of the simple poles are $\pm \gamma(X)$).

  We assume for notational simplicity that $\gamma(X)$ is purely imaginary; we can deduce the general case from this by simply rotating everything.   In the figure, $\beta'$ consists of the red arcs that intersect the shaded annuli.  There are two contributions to $\beta'$.  The first, which we call $\beta_1$ is composed of the two straight segments connecting the circles of radii $1$ and $\sqrt{|t|}$, on either side of the nodes.  The second, $\beta_2$, is an arc of the circle of radius $\sqrt{t}$.  We see that $\operatorname{Re}(\beta'(X)) = \beta_1(X)$ and $i\cdot \operatorname{Im}(\beta'(X))=\beta_2(X)$.  The height of the cylinder is approximately $|\beta_1(X)|$, so its area is approximately $|\gamma(X)|\cdot |\beta_1(X)|$.  Since this must be less than $1$, we have
  $$|\beta_1(X)| \le \frac{1}{|\gamma(X)|}.$$
  To bound $\beta_2(X)$, we note that $\beta_2$ winds around the circle at most once (because of the way $V$ was defined in terms of restrictions of plumbing coordinates).  Suppose its endpoint is at angle $\rho\in [0,2\pi]$.  Then
  $$\beta_2(X) = \int_{\beta_2}\frac{\gamma(X)}{2\pi i} \frac{du}{u} = \frac{\gamma(X)}{2\pi} \int_0^{\rho} d\theta = \frac{\gamma(X)}{2\pi}\rho$$
  is a purely imaginary number of magnitude at most $|\gamma(X)|$.
  
  Summarizing the above, we get the following estimates for the real and imaginary parts of $\beta'(X)$: 
  \begin{align}
    |\operatorname{Re}(\beta'(X))| = |\beta_1(X)| &\le \frac{1}{|\gamma(X)|}, \label{real_beta}\\
    |\operatorname{Im}(\beta'(X))|=|-i\cdot \beta_2(X)| &\le |\gamma(X)|.  \label{imag_beta}
  \end{align}

  Now we bound the remaining piece $\beta''$.  Let $X'=\alpha X$ be the surface in the projective class of $X$ that is produced in the plumbing construction.  
  By assumption $\beta''$ lies on level subsurfaces at or below the level $Y$.  Arguing as in Lemma \ref{lemma:period-bounds} (Period bounds), we see that $|\beta''(X)|\le C \operatorname{size}_X(Y)$.  Now $\operatorname{size}_X(Y)$ is very close to a constant multiple of $\gamma(X')$, since $\gamma$ lies at level $Y$.  Hence $|\beta''(X')|\le C'|\gamma(X')|,$
  where $C'$ is a constant depending on $\beta$ (and $\bar X$).   Since $X$ is just a rescaling of $X'$, we get 
  $$|\beta''(X)|\le C'|\gamma(X)|.$$
  
  Finally, by combining this estimate for $\beta''(X)$ with the estimates \eqref{real_beta}, \eqref{imag_beta} for $\beta'(X)$, we see that the sum $\beta(X)=\beta'(X)+\beta''(X)$ must lie in $\mathcal{R}(\gamma(X))$, for the $\mathcal{R}$ defined at the beginning of the proof.  

  For $w=\gamma(X)$ small, which is the relevant regime, the area of this rectangle $\mathcal{R}(w)$ is bounded above by some $R$ depending only on $C'$.  The property $s\mathcal{R}(w) \subset \mathcal{R}(sw)$ holds for any $w$ and $0< s \le 1$.

\end{proof}

The next lemma bounds the length of a saddle connection from below in terms of the period of any fixed relative homology class that lies at or below the $\succ$-level of the saddle connection.  

\begin{lemma}[Saddle connection bounds]
  \label{lemma:sc-bounds}
  For $\gamma$ a relative homology class defined on the surfaces in $V$, there exists a constant $c>0$ (depending only on $\gamma$ and $V$) with the following property.  For any $X\in V$ and any $\succ \in \mathcal{O}$ consistent with $X$, we let $Y$ be the $\succ$-level of $\gamma$.  Let $s$ be any saddle connection on $X$ whose relative homology class has $\succ$-level at least $Y$. Then 
  $$|s(X)| \ge c \cdot|\gamma(X)|.$$
\end{lemma}

\begin{proof}

  Since the desired inequality is invariant under scaling $X$, we can assume that $X$ is the distinguished surface in the projective equivalence produced by the plumbing construction.
  
  Define the $\emph{injectivity radius}$ $\operatorname{inj}(Y)$ of a subsurface $Y$ (with boundary) of a translation surface to be the infimum of flat length over all curves/arcs (arcs are allowed to have endpoints at singular points or on the boundary) that cannot be homotoped (rel endpoints) to lie in the boundary.
  
  As in the proof of Lemma \ref{lemma:period-bounds} (Period bounds), we divide up $X$ into pieces corresponding to elements of $\mathcal{S}$, together with plumbing annuli.
  
  We claim that
  \begin{align}
    \label{eq:sad-inj}
    |s(X)| \ge \operatorname{inj}(\hat Z),
  \end{align}
  for some subsurface $\hat Z \subset X$ corresponding to an element $Z\in \mathcal{S}$ with $Z\succeq Y$.  In fact, by the definition of $\succ$-level of a relative homology class, $s$ must intersect a subsurface $Z\in \mathcal{S}$ of $\succ$-level at least $Y$ in a curve/arc that cannot be homotoped rel endpoints to lie in subsurfaces at smaller $\succ$-levels.   Hence the length of $s$ on $X$ must be at least the injectivity radius of $\hat Z$.  

Now we claim that 
  \begin{align}
    \label{eq:inj-size}
    \operatorname{inj}(\hat Z)\ge c' \cdot \operatorname{size}_X(Z),
  \end{align}
  where $c'$ is a constant depending only on $\bar X$.
  In fact, if we remove fixed size neighborhoods from the nodes of $Z$ as in the plumbing construction, the resulting surface with boundary has some non-zero injectivity radius $r$.  We then obtain $\hat Z$ from this by scaling by the parameter $t_{-1}^{a_{-1}}\cdots t_i^{a_i}$, adding a small modification differential, and then merging zeros.  The latter two operations only change the injectivity radius by a small amount, so the injectivity radius of $\hat Z$ is close to $r\cdot |t_{-1}^{a_{-1}}\cdots t_i^{a_i}|$, and $|t_{-1}^{a_{-1}}\cdots t_i^{a_i}|$ is exactly $\operatorname{size}_X(Z)$, so \eqref{eq:inj-size} follows.

  Now $Z$ was chosen such that $Z\succeq Y$, and since $X$ is assumed to be $\succ$-consistent, we have
  \begin{align}
    \label{eq:z-y}
    \operatorname{size}_X(Z) \ge \operatorname{size}_X(Y).
  \end{align}
  
  Finally, from the proof of Lemma \ref{lemma:period-bounds} (specifically inequality \eqref{eq:xprime}), using that our $X$ is the distinguished representative of its projective equivalence class, we get that
  \begin{align}
    \label{eq:size-gamma}
    \operatorname{size}_X(Y) \ge \frac{1}{C} \cdot |\gamma(X)|,
  \end{align}
  for some $C$ depending only on $\gamma$ and $V$.
  
    Stringing together the inequalities \eqref{eq:sad-inj}, \eqref{eq:inj-size}, \eqref{eq:z-y}, \eqref{eq:size-gamma} gives the desired inequality, for some $c$ depending only on $\bar X$. 

\end{proof}

\subsection{Choosing a basis}
In this section we describe how to choose a basis of $H_1(X,\Sigma;\Z)$, for surfaces $X\in V$, that is adapted to the surfaces in $V$ in the sense that homology classes living on smaller subsurfaces will generally come earlier in the basis.  The basis will depend on an ordering $\succ \in \mathcal{O}$ of subsurfaces $\mathcal{S}$.

Given $Z\in \mathcal{S}$, we define $H_{\preceq Z} \subset H_1(X,\Sigma;\C)$ to be the span of all elements $\gamma \in H_1(X,\Sigma;\Z)$ whose $\succ$-level is $\preceq Z$.

\begin{lemma}[Basis]
  \label{lemma:basis}
  Fix an ordering $\succ\in \mathcal{O}$. There is a basis $\alpha_1,\ldots,\alpha_{m'}$ for $H_1(X,\Sigma;\C)$ consisting of integral classes with the following properties.
    \begin{enumerate}[(i)]
  \item For each $Z$ a basis for the vector space $H_{\preceq Z}$ is given by $\alpha_1,\ldots,\alpha_j$ for some $j$.
  \item The number of crossings of each $\alpha_i$ with the set of $\succ$-wide cylinder circumference curves is at most $1$.  \label{item:wide}
  \item If $\alpha_i$ intersects a $\succ$-wide cylinder circumference $\gamma$, then $\alpha_i$ has a representative such that the level subsurfaces which the representative intersects lie at or below the $\succ$-level of $\gamma$ (of course $\alpha_i$ also crosses the cylinder, which may be at a higher $\succ$-level, but the cylinder is not a level subsurface).  \label{item:cross}
  \end{enumerate}
\end{lemma}

\begin{proof}
  Begin by choosing an integral basis for $H_{\preceq \bar X^{(-N)}}$ (note that for any $\succ \in \mathcal{O}$, $\bar X^{(-N)}$ is always the $\prec$-smallest element of $\mathcal{S}$).  Then continue by appending elements that together with the previously added elements give a basis for $H_{\preceq Z}$, where $Z$ is the second $\prec$-smallest element of $\mathcal{S}$.  Continue in this manner through $Z=\bar X^{(0)}$, the largest level subsurface.  None of the $\alpha_i$ added up to this point will cross $\succ$-wide cylinder circumference curves, so (\ref{item:wide}) and (\ref{item:cross}) are satisfied for these $\alpha_i$.

  By definition of $\succ$-wide, every $Z$ with $Z\succ \bar X^{(0)}$ is a $\succ$-wide cylinder.  Each $\succ$-wide cylinder joins two components, possibly equal.  We can find zeros of the differential by moving to either side of the cylinder and then moving down level subsurfaces until we come to a locally minimal component in the level graph (such a component must have a zero).  We connect these zeros (which might coincide), giving a relative homology class $\alpha_i$ that satisfies (\ref{item:wide}) and  (\ref{item:cross}).  We add this $\alpha_i$ to the basis.  We do this one by one for each new $\succ$-wide cylinder cross curve whose circumference curve is independent of the current basis elements.  
\end{proof}

\subsection{Volume for a single linear subspace}

Using the lemmas established in the previous two sections, we can now prove the volume estimate for a single linear subspace.

\begin{lemma}[Volume for single linear subspace]
  \label{lemma:vol-subspace}
Let $W$ be a linear subspace of a $V$ from Section \ref{sec:vk}.  Then 
$$\mu_W^1\left(L_{\epsilon_1,\ldots,\epsilon_k}^W \right) = O(\epsilon_1^2 \cdots \epsilon_k^2).$$
\end{lemma}

\begin{proof}

  Our first goal is to choose a basis for linear functionals on $W$ which we will use to do the volume estimation.  Fix an ordering $\succ$, and use Lemma \ref{lemma:basis} (Basis) to choose a basis $\alpha_1,\ldots,\alpha_{m'}$ for $H_1(X,\Sigma;\C)$ satisfying the three conditions given by that lemma. 

The choice of $\{\alpha_i\}$ basis above was made independently of the subspace $W$.  Now extract a basis $\{\beta_1,\ldots,\beta_m\} \subset \{\alpha_1,\ldots,\alpha_{m'}\}$ for $W^*$ as follows.  
Let $\beta_1=\alpha_1$, and suppose we have chosen $\beta_1,\ldots,\beta_j$.  Take $\beta_{j+1}$ to be $\alpha_i$, where $i$ is the smallest index such that $\alpha_i,\beta_1,\ldots,\beta_j$ give linearly independent functionals on the subspace $W$.  Continue doing this until no additional $\alpha_i$ can be added satisfying the conditions.  Note that by construction, for each $Z\in \mathcal{S}$ the subspace of $W^*$ given by elements of $H_{\preceq Z}$ has basis $\beta_1,\ldots,\beta_j$ for some $j$.

Our next goal is to show that certain sets defined by inequalities on the $\beta_i$ functionals contain all the surfaces in $L_{\epsilon_1,\ldots,\epsilon_k}^W$ that are $\succ$-consistent.  

\medskip \noindent {\bf Claim:}
If $X$ is $\succ$-consistent and $X\in L_{\epsilon_1,\ldots,\epsilon_k}^W$, then there exists a bijection $\phi:\{1,\ldots, k\}\to \{1,\ldots, k\}$ (i.e. a permutation) such that  $X$ is in the set
\begin{align*}  
   B^{\phi, \succ}_{\epsilon_1,\ldots,\epsilon_k}:=\big\{ &X'\in W^1: X' \text{ is } \succ\text{-consistent}, \text{ } \\
                                                         &|\beta_{\phi(1)}(X')|\le \epsilon_1/c, \ldots, |\beta_{\phi(k)} (X')|\le \epsilon_k/c \big\},
\end{align*}
where $c>0$ is the minimum of the constants from Lemma \ref{lemma:sc-bounds} (Saddle connection bounds) applied with $\gamma=\beta_1,\ldots,\beta_k$.

\medskip \noindent\emph{Proof of Claim:}  In fact, by definition of $L_{\epsilon_1,\ldots,\epsilon_k}^W$, such an $X$ must have saddle connections $s_1,\ldots,s_k$ that are $W$-independent and such that $|s_i(X)|\le \epsilon_i$.   Because of the way the $\beta_i$ were chosen, for each $s_i$ we can find an element of $\beta_{\phi(i)}$ such that
  \begin{enumerate}[(a)]
  \item the $\succ$-level of $s_i$ is at least that of $\beta_{\phi(i)}$, and
  \item all the $\phi(i)$ are distinct.
  \end{enumerate}
  (Concretely, if the $s_i$ are in increasing order in terms of $\succ$-level, we can simply take $\phi(i)=i$; otherwise, the situation is just a permutation of this).  Now by Lemma \ref{lemma:sc-bounds} (Saddle connection bounds), we have $\epsilon_i \ge |s_i(X)| \ge c |\beta_{\phi(i)}(X)|$, so $X$ is in the set $B^{\phi, \succ}_{\epsilon_1,\ldots,\epsilon_k}$.  This completes the proof of the Claim.  

  \medskip
  
  Our goal now is to establish the inclusion (\ref{eq:inclusion}) below, which we will then use to estimate the volume of $B^{\phi, \succ}_{\epsilon_1,\ldots,\epsilon_k}$.
  The definition of $B^{\phi, \succ}_{\epsilon_1,\ldots,\epsilon_k}$ gives bounds on $\beta_1,\ldots,\beta_k$; we now control the rest of the $\beta_i$ for $X'$ that are $\succ$-consistent.  
  Roughly, some of these $\beta_i$ do not cross any $\succ$-wide cylinders; for these $\beta_i$, the absolute value is bounded.
  The remaining $\beta_i$ can be long because they cross $\succ$-wide cylinders, but nevertheless, after fixing the other periods, such a $\beta_i$ lies in a set of bounded area. 
  
  To formalize the above, we find the index $\ell$ such that all $\beta_i$ with $i\le \ell$ lie at or below $\succ$-level $\bar X^{(0)}$ and all $\beta_i$ with $i>\ell$ lie above $\succ$-level $\bar X^{(0)}$ (and hence cross $\succ$-wide cylinders).
  
  First consider the $\beta_i$ for $i\le \ell$.  Using Lemma \ref{lemma:period-bounds} (Period bounds) and the fact that, over surfaces in $V$, the size of any level subsurface is bounded from above, we get that there exists a $K$ such that for any $X'\in W^1$ consistent with $\succ$, 
  \begin{align}
    \label{eq:small-i} 
    |\beta_i(X')|< K,
  \end{align}
when $i\le \ell$.  
  
  Next we consider $\beta_i$ for $i>\ell$.  Such a $\beta_i$ crosses exactly one $\succ$-wide cylinder circumference curve $\gamma_i$, and since (\ref{item:cross}) of Lemma \ref{lemma:basis} (Basis) also holds, all hypotheses of Lemma \ref{lemma:cyl-bounds} (Cylinder bounds) are satisfied.  
  Hence by that lemma, $\beta_i(X') \in \mathcal{R}(\gamma_i(X'))$, where $\mathcal{R}$ is a rectangle of area at most $R$. 
  Now there is some linear function $f_i$ such that $\gamma_i(X') = f_i(\beta_1(X'),\ldots,\beta_{\ell}(X'))$ for all $X'\in W^1$.
  Let $\mathcal{R}_i(z_1,\ldots,z_\ell):=\mathcal{R}(f_i(z_1,\ldots,z_\ell)). $ So 
    \begin{align}
      \label{eq:large-i}
      \beta_i(X') \in \mathcal{R}_{i}(\beta_1(X'),\ldots,\beta_{\ell}(X')),
    \end{align}
    when $i>\ell$. Furthermore, by the inclusion property given by Lemma \ref{lemma:cyl-bounds} (Cylinder bounds), for any $s\le 1$,
  \begin{align}
    s\mathcal{R}_i(z_1,\ldots,z_\ell) \subset \mathcal{R}_i(sz_1,\ldots,sz_\ell)). \label{eq:R}
  \end{align}

  Now putting the cases \eqref{eq:small-i} and \eqref{eq:large-i} together yields the desired inclusion:
  \begin{equation}
         \label{eq:inclusion}
    \begin{aligned}
    B^{\phi, \succ}_{\epsilon_1,\ldots,\epsilon_k} \subset & \big\{ X'\in W^1 : |\beta_{\phi(1)}(X')|\le \epsilon_1/c, \ldots, |\beta_{\phi(k)} (X')|\le \epsilon_k/c, \\    
    &|\beta_{k+1}(X')| \le K, \ldots, |\beta_{\ell}(X')| \le K, \\
    &\beta_{\ell+1}(X') \in \mathcal{R}_{\ell+1}\left(\beta_1(X'),\ldots,\beta_{\ell}(X')\right), \ldots, \beta_{j}(X') \in \mathcal{R}_{j}\left(\beta_1(X'),\ldots,\beta_{\ell}(X')\right)\big\}.
    \end{aligned}
  \end{equation}

  Now we can compute $\mu_W$ volume, up to a constant factor, by taking standard Lebesgue volume $\lambda$ on $\C^j$ with respect to any basis.  In particular, we can work with the basis $z_1,\ldots,z_j$ such that the $\beta_1,\ldots,\beta_j$ are the coordinate functions with respect to this basis.  We then compute using \eqref{eq:inclusion} and \eqref{eq:R}:
  
  \begin{align*}
    \mu_W^1(B^{\phi, \succ}_{\epsilon_1,\ldots,\epsilon_k} ) = & \mu_W\left( \left\{ sX': 0\le s \le 1, X'\in B^{\phi, \succ}_{\epsilon_1,\ldots,\epsilon_k}\right\} \right) \\
    \le & \lambda \big(\big\{(sz_1,\ldots,sz_j):  0\le s \le 1, \enskip |z_{\phi(1)}|\le \epsilon_1/c, \ldots, |z_{\phi(k)} |\le \epsilon_k/c, \enskip \\
                                                               & \quad \enskip |z_{k+1}|, \ldots, |z_{\ell}| \le K, \\
                                                               & \quad \enskip  z_{\ell+1} \in \mathcal{R}_{\ell+1}(z_1,\ldots,z_{\ell}), \ldots, z_{j} \in \mathcal{R}_{j}(z_1,\ldots,z_{\ell})  \big\}\big) \\
    = & \lambda \big(\big\{(sz_1,\ldots,sz_j):  0\le s \le 1, \enskip |z_{\phi(1)}|\le \epsilon_1/c, \ldots, |z_{\phi(k)} |\le \epsilon_k/c, \\
                                                               & \quad \enskip |z_{k+1}|, \ldots, |z_{\ell}| \le K,\\ 
                                                               &\quad \enskip sz_{\ell+1} \in s\mathcal{R}_{\ell+1}(z_1,\ldots,z_{\ell}), \ldots, sz_{j} \in s\mathcal{R}_{j}(z_1,\ldots,z_{\ell})  \}) \\ %
    \le & \lambda \big (\big\{(sz_1,\ldots,sz_j): 0\le s \le 1, \enskip |sz_{\phi(1)}|\le \epsilon_1/c, \ldots, |sz_{\phi(k)}|\le \epsilon_k/c, \\
                                                               &\quad \enskip |s z_{k+1}|, \ldots, |s z_{\ell}| \le K, \\
                                                               & \quad \enskip  sz_{\ell+1} \in \mathcal{R}_{\ell+1}(sz_1,\ldots,sz_{\ell}), \ldots, sz_{j} \in \mathcal{R}_{j}(sz_1,\ldots,sz_{\ell})  \big\}\big) \\
    = & \lambda \big(\big\{(z_1,\ldots,z_j): |z_{\phi(1)}|\le \epsilon_1/c, \ldots, |z_{\phi(k)} |\le \epsilon_k/c, \\
                                                               & \quad \enskip |z_{k+1}|, \ldots, |z_{\ell}| \le K, \\
                                                               & \quad \enskip  z_{\ell+1} \in \mathcal{R}_{\ell+1}(z_1,\ldots,z_{\ell}), \ldots, z_{j} \in \mathcal{R}_{j}(z_1,\ldots,z_{\ell})  \big\}\big) \\
    = &\int_{B(\epsilon_1/c)} \cdots \int_{B(\epsilon_k/c)}  \int_{B(K)} \cdots \int_{B(K)} \\
                                                               & \int_{\mathcal{R}_{\ell+1}(z_1,\ldots,z_{\ell})} \cdots \int_{\mathcal{R}_{j}(z_1,\ldots,z_{\ell})} dz_j\wedge d\bar z_j  \cdots dz_{1} \wedge d\bar z_{1} \\
    =& O\left(\epsilon_1^2 \cdots \epsilon_k^2 \cdot K^{\ell-k} \cdot R^{j-\ell}\right)\\
    = & O\left(\epsilon_1^2 \cdots \epsilon_k^2\right).
  \end{align*}

  Since every $X \in L^W_{\epsilon_1,\ldots,\epsilon_k}$ is in some $B^{\phi, \succ}_{\epsilon_1,\ldots,\epsilon_k}$ for one of the finitely many choices of $(\phi,\succ)$, adding up the estimates above over all the $(\phi,\succ)$ gives the desired inequality.
  \end{proof}

\subsection{Proof of local volume bound}
\label{sec:proof-local}

In the previous section we established the volume bound for a single linear subspace.  We now combine this with the facts established in Section~\ref{sec:semianalytic-pieces} to prove Proposition \ref{prop:local-regularity} (Local volume bound).
\begin{proof}[Proof of  Proposition \ref{prop:local-regularity} (Local volume bound)]
  By Lemma \ref{lemma:semianalytic}, we get a finite collection of semianalytic sets $\P V_i$ that cover $U\cap \P\H$, where $U$ is some neighborhood of $\bar X$ in $\barPH$.  Hence $p^{-1}(U) \subset \bigcup V_i$, and so
   $$\mu_\M^1 \left( L_{\epsilon_1,\ldots,\epsilon_k}^\M \cap p^{-1}(U)\right) \le \sum_i \mu_\M^1 \left( L_{\epsilon_1,\ldots,\epsilon_k}^\M \cap V_i\right).$$

  Let $\{W_{\ell,i}\}$ be the finite collection of all the linear subspaces $W_\ell$ that we get from applying  Lemma \ref{lemma:finite-subspaces} to each of the sets $V_i$.  After summing, that lemma gives 
   \begin{align*}
     \sum_i \mu_\M^1 \left( L_{\epsilon_1,\ldots,\epsilon_k}^\M \cap V_i\right) & \le C\cdot \sum_{\ell,i}\mu_{W_{\ell,i}}^1\left(L_{\epsilon_1,\ldots,\epsilon_k}^{W_{\ell,i}}\right),
   \end{align*}
   for some constant $C$.  

  Then applying Lemma \ref{lemma:vol-subspace} to each subspace $W_{\ell,i}$, we get
  \begin{align*}
    \sum_{\ell,i}\mu_{W_{\ell,i}}^1\left(L_{\epsilon_1,\ldots,\epsilon_k}^{W_{\ell,i}}\right) & = O(\epsilon_1^2 \cdots \epsilon_k^2).
  \end{align*}

  Combining the three inequalities above gives
  $$\mu_\M^1 \left( L_{\epsilon_1,\ldots,\epsilon_k}^\M\cap p^{-1}(U)\right) =O(\epsilon_1^2 \cdots \epsilon_k^2) ,$$
    as desired.  
  
\end{proof}

\section{Proof of Theorem \ref{thm:regularity}}
\label{sec:proof}
To prove the main theorem, all that remains is to combine the local volume estimate with compactness of $\barPH$.  

\begin{proof}[Proof of Theorem \ref{thm:regularity} given Proposition \ref{prop:local-regularity} (Local volume bound)]
    We define an open cover $\{U_X\}_{X\in \barPH}$ of $\barPH$ as follows.  For a boundary point $X\in \barPH - \P\H$, take $U_X$ to be the open set given by Proposition \ref{prop:local-regularity} (Local volume bound).  For a point $X\in \P\H$, we take $U_X$ to be any open neighborhood of $X$ contained in some compact subset of $\P\H$.  

    We have an open cover of the set $\barPH$, which is a compact space (Theorem \ref{thm:compact}).  Thus there is a finite subcover $U_{X_1,}, \ldots, U_{X_n}$.  For the $X_i$ in this list that are in the boundary, we use the volume estimate given in Proposition \ref{prop:local-regularity} (Local volume bound).  We take $C'$ to be the maximum of the $C$'s given by that proposition over the different $X_i$ in the boundary.  So for such $i$, 
  $$\mu_\M^1 (p^{-1}(U_{X_i})\cap L_{\epsilon_1,\ldots,\epsilon_k}^\M) \le C' \epsilon_1^2 \cdots \epsilon_k^2.$$
  For the $X_i$ not in the boundary, $U_{X_i}$ is contained in some compact set $K$, and all surfaces in $K$ have a uniform lower bound on saddle connection length.  Thus if any of the $\epsilon_i$ are small, the above estimate automatically holds, since the left hand side is zero.  In the case where none of the $\epsilon_i$ are small, we can enlarge $C'$ so that the estimate holds (we can do this without knowing a priori that $\mu_\M^1$ is a finite measure - in fact $\mu^1_\M(p^{-1}(U_{X_i})) \le \mu^1_\M(p^{-1}(K))$ which is finite, since $\mu^1_\M$ is locally finite and $p^{-1}(K)$ is compact).  Since $\bigcup U_{X_i}$ contains $\P\H$, we have that $\bigcup p^{-1}(U_{X_i})$ contains $\H$.  Adding everything up gives
  $$\mu_\M^1 (L_{\epsilon_1,\ldots,\epsilon_k}^\M) \le nC' \epsilon_1^2 \cdots \epsilon_k^2,$$
\end{proof}

\bibliography{sources}{}

\newcommand{\etalchar}[1]{$^{#1}$}
\providecommand{\bysame}{\leavevmode\hbox to3em{\hrulefill}\thinspace}
\providecommand{\MR}{\relax\ifhmode\unskip\space\fi MR }
\providecommand{\MRhref}[2]{%
  \href{http://www.ams.org/mathscinet-getitem?mr=#1}{#2}
}
\providecommand{\href}[2]{#2}
\begin{thebibliography}{BCG{\etalchar{+}}22}

\bibitem[ACM19]{acm2019}
Jayadev~S. Athreya, Yitwah Cheung, and Howard Masur, \emph{Siegel-{V}eech
  transforms are in {$L^2$}}, J. Mod. Dyn. \textbf{14} (2019), 1--19, With an
  appendix by Athreya and Rene R\"{u}hr. \MR{3959354}

\bibitem[AMY13]{amy2013}
Artur Avila, Carlos Matheus, and Jean-Christophe Yoccoz,
  \emph{{$SL(2,\Bbb{R})$}-invariant probability measures on the moduli spaces
  of translation surfaces are regular}, Geom. Funct. Anal. \textbf{23} (2013),
  no.~6, 1705--1729. \MR{3132901}

\bibitem[BCG{\etalchar{+}}18]{bcggm2018}
Matt Bainbridge, Dawei Chen, Quentin Gendron, Samuel Grushevsky, and Martin
  M\"{o}ller, \emph{Compactification of strata of {A}belian differentials},
  Duke Math. J. \textbf{167} (2018), no.~12, 2347--2416. \MR{3848392}

\bibitem[BCG{\etalchar{+}}22]{bcggm2019}
Matt {Bainbridge}, Dawei {Chen}, Quentin {Gendron}, Samuel {Grushevsky}, and
  Martin {M{\"o}ller}, \emph{{The moduli space of multi-scale differentials}},
  arXiv e-prints (2022), arXiv:1910.13492v2.

\bibitem[BM88]{bm1988}
Edward Bierstone and Pierre~D. Milman, \emph{Semianalytic and subanalytic
  sets}, Inst. Hautes \'{E}tudes Sci. Publ. Math. (1988), no.~67, 5--42.
  \MR{972342}

\bibitem[CW19]{cw2019}
Dawei {Chen} and Alex {Wright}, \emph{{The WYSIWYG compactification}}, arXiv
  e-prints (2019), arXiv:1908.07436.

\bibitem[EKZ14]{ekz2014}
Alex Eskin, Maxim Kontsevich, and Anton Zorich, \emph{Sum of {L}yapunov
  exponents of the {H}odge bundle with respect to the {T}eichm\"uller geodesic
  flow}, Publ. Math. Inst. Hautes \'Etudes Sci. \textbf{120} (2014), 207--333.
  \MR{3270590}

\bibitem[EM18]{em2018}
Alex Eskin and Maryam Mirzakhani, \emph{Invariant and stationary measures for
  the {${\rm SL}(2,\Bbb R)$} action on moduli space}, Publ. Math. Inst. Hautes
  \'{E}tudes Sci. \textbf{127} (2018), 95--324. \MR{3814652}

\bibitem[Fil16]{filip2016}
Simion Filip, \emph{Splitting mixed {H}odge structures over affine invariant
  manifolds}, Ann. of Math. (2) \textbf{183} (2016), no.~2, 681--713.
  \MR{3450485}

\bibitem[Gad17]{gadre2017}
Vaibhav Gadre, \emph{Partial sums of excursions along random geodesics and
  volume asymptotics for thin parts of moduli spaces of quadratic
  differentials}, J. Eur. Math. Soc. (JEMS) \textbf{19} (2017), no.~10,
  3053--3089. \MR{3713001}

\bibitem[{Gad}19]{gadre2019}
Vaibhav {Gadre}, \emph{{Trimmed sums of twists and the area Siegel-Veech
  constant}}, arXiv e-prints (2019), arXiv:1909.02086.

\bibitem[{Kan}11]{kankaanrinta2011}
Marja {Kankaanrinta}, \emph{{On subanalytic subsets of real analytic
  orbifolds}}, arXiv e-prints (2011), arXiv:1104.4653.

\bibitem[MS91]{ms1991}
Howard Masur and John Smillie, \emph{Hausdorff dimension of sets of nonergodic
  measured foliations}, Ann. of Math. (2) \textbf{134} (1991), no.~3, 455--543.
  \MR{1135877}

\bibitem[MW17]{mw2017}
Maryam Mirzakhani and Alex Wright, \emph{The boundary of an affine invariant
  submanifold}, Invent. Math. \textbf{209} (2017), no.~3, 927--984.
  \MR{3681397}

\bibitem[{Ngu}12]{nguyen2012}
Duc-Manh {Nguyen}, \emph{{Volumes of the sets of translation surfaces with
  small saddle connections in rank one affine submanifolds}}, arXiv e-prints
  (2012), arXiv:1211.7314.

\bibitem[{Ngu}19]{nguyen2019}
\bysame, \emph{{Volume form on moduli spaces of d-differentials}}, arXiv
  e-prints (2019), arXiv:1902.04830.

\bibitem[Str84]{strebel1984}
Kurt Strebel, \emph{Quadratic differentials}, Ergebnisse der Mathematik und
  ihrer Grenzgebiete (3) [Results in Mathematics and Related Areas (3)],
  vol.~5, Springer-Verlag, Berlin, 1984. \MR{743423}

\bibitem[Wri15a]{wright2015}
Alex Wright, \emph{Cylinder deformations in orbit closures of translation
  surfaces}, Geom. Topol. \textbf{19} (2015), no.~1, 413--438. \MR{3318755}

\bibitem[Wri15b]{wrightSurvey}
\bysame, \emph{Translation surfaces and their orbit closures: an introduction
  for a broad audience}, EMS Surv. Math. Sci. \textbf{2} (2015), no.~1,
  63--108. \MR{3354955}

\end{thebibliography}
\bibliographystyle{amsalpha}

\end{document}